\documentclass[12pt]{amsart} 
\usepackage{amsfonts,graphics,amsmath,amsthm,amsfonts,
amscd,amssymb,amsmath,enumitem}

\usepackage{soul}

\newcommand{\factor}[2]{\left. \raise 2pt\hbox{\ensuremath{#1}} \right/
        \hskip -2pt\raise -2pt\hbox{\ensuremath{#2}}}

\usepackage{soul}
\usepackage[all]{xy}
\usepackage{color}
\usepackage{hyperref}
\RequirePackage[dvipsnames,usenames]{xcolor}

\hypersetup{
bookmarks,
bookmarksdepth=3,
bookmarksopen,
bookmarksnumbered,
pdfstartview=FitH,
colorlinks,backref,hyperindex,
linkcolor=Sepia,
anchorcolor=BurntOrange,
citecolor=MidnightBlue,
citecolor=OliveGreen,
filecolor=BlueViolet,
menucolor=Yellow,
urlcolor=OliveGreen
}

\usepackage[left=1.02in,top=1.0in,right=1.02in,bottom=1.0in]{geometry}

\usepackage{xparse}
\usepackage{tikz}
\usepackage{adjustbox}
\makeatletter
\newcommand{\uset}[3][0ex]{%
  \mathrel{\mathop{#3}\limits_{
    \vbox to#1{\kern-3\ex@
    \hbox{$\scriptstyle#2$}\vss}}}}
\makeatother

\NewDocumentCommand\RightArrow{O{2.0ex} O{black}}{%
  \mathrel{ \tikz[baseline] 
   \draw [line width=0.5pt, decoration={markings,mark=at position 1 with {\arrow[scale=2, line width=0.25pt]{to}}},  postaction={decorate}, #2] (0,0pt) -- ++(#1,0pt);}
   }
\newcommand{\hocolim}{\mathop{\uset{\RightArrow[35pt]}{\rm hocolim}}		}
\newcommand{\holim}{\mathop{\uset{\RightArrow[-28pt]}{\rm holim}}		}

\theoremstyle{plain}
\newtheorem{theorem}{Theorem}[section]
\newtheorem{assumption}[theorem]{Assumption}
\newtheorem{corollary}[theorem]{Corollary}
\newtheorem{lemma}[theorem]{Lemma}
\newtheorem{claim}[theorem]{Claim}

\newtheorem{proposition}[theorem]{Proposition}
\newtheorem{definition-lemma}[theorem]{Definition-Lemma}
\newtheorem{question}[theorem]{Question}

\newtheorem{definition}[theorem]{Definition}

\theoremstyle{definition}
\newtheorem{remark}[theorem]{Remark}
\newtheorem{example}[theorem]{Example}
\newtheorem{notation}[theorem]{Notation}

\def\e{\Cal E}

\def\e1{E_1}
\def\e2{E_2}

\def\OO{\mathcal O}

\def\QQ{\mathcal Q}

\newcommand{\bN}{{\mathbb{N}}}
\newcommand{\bZ}{{\mathbb{Z}}}
\newcommand{\sA}{\mathcal{A}}
\newcommand{\sC}{\mathcal{C}}
\newcommand{\sD}{\mathcal{D}}
\newcommand{\sE}{\mathcal{E}}
\newcommand{\sF}{\mathcal{F}}
\newcommand{\sK}{\mathcal{K}}
\newcommand{\sG}{\mathcal{G}}
\newcommand{\sH}{{\mathcal{H}}}
\newcommand{\sM}{\mathcal{M}}
\newcommand{\sO}{{\mathcal{O}}}
\newcommand{\sP}{{\mathcal{P}}}

\DeclareMathOperator{\id}{id}
\DeclareMathOperator{\Hom}{Hom}
\DeclareMathOperator{\im}{im}

\DeclareMathOperator{\pr}{pr}
\DeclareMathOperator{\codim}{codim}
\DeclareMathOperator{\Ext}{Ext}
\DeclareMathOperator{\depth}{depth}
\DeclareMathOperator{\Spec}{Spec}
\DeclareMathOperator{\Supp}{Supp}
\DeclareMathOperator{\coker}{coker}
\DeclareMathOperator{\Id}{Id}

\usepackage{tikz}
\usetikzlibrary{decorations.markings}
\usepackage{xparse}
\usepackage{mathtools}
\usepackage{minibox}

\NewDocumentCommand\DownArrow{O{2.0ex} O{black}}{%
   \mathrel{\tikz[baseline] \draw [<-, line width=0.5pt, #2] (0,0) -- ++(0,#1);}
}

\NewDocumentCommand\UpArrow{O{2.0ex} O{black}}{%
   \mathrel{\tikz[baseline] \draw [line width=0.5pt, decoration={markings,mark=at position 1 with {\arrow[scale=2, line width=0.25pt]{to}}},  postaction={decorate}, #2] (0,0) -- ++(0,#1);}
}

\newcommand{\expl}[2]{\underset{\mathclap{\minibox[c]{$\UpArrow[10pt]$\\ \fbox{\footnotesize #2}}}}{#1}}

\newcommand{\explshift}[3]{\underset{\mathclap{\minibox[c]{$\UpArrow[10pt]$\\ \hspace{#1} \fbox{\footnotesize #3}}}}{#2}}

\newcommand{\explabove}[2]{\overset{\mathclap{\minibox[c]{ \fbox{\footnotesize #2} \\ $\DownArrow[10pt]$}}}{#1}}

\newcommand{\explaboveparshift}[4]{\overset{\mathclap{\minibox[c]{ \hspace{#2} \fbox{\parbox{#1}{\footnotesize #4}} \\ $\DownArrow[10pt]$}}}{#3}}

\newcommand{\explparshift}[4]{\underset{\mathclap{\minibox[c]{$\UpArrow[10pt]$\\ \hspace{#2} \fbox{\parbox{#1}{\footnotesize #4}}}}}{#3}}

\newcommand{\ass}{\sA ss}

\author{Christopher D. Hacon}
\address{Department of Mathematics \\
University of Utah\\
Salt Lake City, UT 84112, USA}
\email{hacon@math.utah.edu}
\author{Zsolt Patakfalvi}
\address{\'Ecole Polytechnique F\'ed\'erale de Lausanne (EPFL), MA C3 635, Station 8, 1015 Lausanne, Switzerland}
\email{zsolt.patakfalvi@epfl.ch}

\begin{document}
\title{Generic vanishing in characteristic $p>0$ and the geometry of theta divisors}

\begin{abstract}
In this paper we prove a strengthening of the generic vanishing result in characteristic $p>0$ given in \cite{HP16}. As a consequence of this result, we show that irreducible $\Theta$ divisors are strongly F-regular and we prove a related result for pluri-theta divisors.
\end{abstract}
\maketitle
\bigskip
\begin{center}{\it To Fabrizio Catanese on the occasion of his 70-th birthday}
\end{center}
\tableofcontents
\section{Introduction} 

Since the work of Green and Lazarsfeld \cite{GL90} and \cite{GL91} it has been clear that generic vanishing plays a fundamental role in the geometry of complex projective irregular varieties, i.e. varieties such that the Albanese map is non-trivial. Combining these results with the Fourier-Mukai transform \cite{Mukai81} and Koll\'ar's results on (higher) direct images of dualizing sheaves \cite{Kollar86I,Kollar86II}, many surprisingly far reaching and precise results on the birational geometry of complex projective varieties have been proven over the last couple of decades. In \cite{Hacon04} a new more functorial perspective was introduced and in particular it was shown that generic vanishing is a consequence of traditional vanishing theorems and Koll\'ar's results above on higher direct images. Unluckily, it is well known that (generic) vanishing theorems and the results of
\cite{Kollar86I,Kollar86II} fail in positive characteristic  \cite{HK15}. Nevertheless recently, using the theory of Cartier modules and Serre vanishing, a technical generalization of generic vanishing was proven in \cite{HP16} for projective varieties over a field of characteristic $p>0$.
Despite its technical nature, this result leads the way to some remarkable geometric applications such as a characterization of (ordinary) abelian varieties \cite{HPZ19}. 
In this paper we further refine the results of \cite{HP16} proving a more precise generic vanishing result in characteristic $p>0$ and we use this result to investigate the geometry of pluri-theta divisors. 


\subsection{Generic vanishing theorems} 

\subsubsection{Characteristic zero background}
Let $X$ be a complex projective manifold and $a:X\to A$ its Albanese morphism. Green and Lazarsfeld study the cohomological support loci 
\[V^i(\omega _X):=\big\{ \ P\in {\rm Pic }^0(A)\ \big|\ h^i(\omega _X\otimes P)\ne 0\ \big\}.\]
In \cite{GL90} and \cite{GL91}, they show that these loci are translates of abelian subvarieties of ${\rm Pic }^0(A)$ of codimension $\geq i-(\dim X-\dim a(X))$. In particuar if $X$ is of maximal Albanese dimension so that $\dim X=\dim a(X)$, then ${\rm codim}V^i(\omega _X)\geq i$.
They also conjecture the following more functorial version of this result  (see \S \ref{S-P} for the definition of the Grothendieck dual $D_A$ and the Fourier Mukai transform $R\hat S$):
\begin{equation}
\label{eq:intro_functorial}
\parbox{330pt}{ if $X$ is of maximal Albanese dimension, then $R\hat S (D_A(Ra_*\omega _X))$ is a sheaf, or equivalently $\mathcal H ^i\Big( R\hat S \big( D_A(R a_* \omega _X) \big) \Big)=0$ for $i\ne 0$.}
\end{equation}
In \cite[Theorem 1.2.3]{Hacon04} a stronger version of conjecture \eqref{eq:intro_functorial} was proven. Namely it is shown that if $F$ is a coherent sheaf on $A$ and $L$ is an ample line bundle on $\hat A$ such that $H^i\big(A,F\otimes \hat L ^\vee \big)=0$ for any $i>0$ and any sufficiently ample line bundle $L$, then  $R\hat S \big( D_A(F) \big)$ is a sheaf (see \S \ref{S-P} for the definition of $\hat L ^\vee$).
Note that by \cite{Kollar86II}, \[Ra_*\omega _X=\sum _{j=0}^{\dim X-\dim a(X)}R^ja_*\omega _X[-j]\] and by \cite{Kollar86I}, \[H^i\big(A,R^ja_*\omega _X\otimes \hat L ^\vee \big)=0,\qquad {\rm for }\ i>0.\]
In particular this implies \eqref{eq:intro_functorial}, which in turn implies  $\codim \mathcal H ^{j}\big(R\hat S(Ra_* \omega _X)\big) \geq j - (\dim X - \dim a(X)$. The original generic vanishing results then follow easily.

\subsubsection{Positive characteristic background}
In characteristic $p>0$, it is known that the results of  \cite{Kollar86I} and \cite{Kollar86II} fail frequently. For example, there exist smooth surfaces such that $H^1(\omega _X\otimes A)\ne 0$ for an ample line bundle $A$ and there are examples where the sheaves $R^ia_*\omega _X$ are torsion (even when $a$ is surjective). Thus, traditional generic vanishing results must fail as shown in \cite{HK15}. On the other hand it is well known that one can recover some weaker vanishing results by combining Serre Vanishing and the Frobenius morphism.
For example, note that if $f:X\to Y$ is a projective morphism of smooth projective varieties then $R^if_* \omega _X$ is a \emph{Cartier module}. By definition this means that   we have an $\sO_X$-module homomorphism $F_*R^if_*\omega _X\to R^if_*\omega _X$, which in this case is given by applying $R^i f_* (\_)$ to the Grothendieck trace of the  Frobenius morphism of $X$. Let $\Omega _e:=F^e_*R^if_*\omega _X$ and  let $H$ be an ample line bundle on $A$. Then, by the projection formula and Serre vanishing we have \[H^j(A,\Omega _e\otimes H)\cong H^j\Big( A,F^e_*\big(R^if_*\omega _X\otimes H^{p^e}\big)\Big)=0\qquad j>0.\]
Since the inverse system $\Omega _e$ satisfies the Mittag-Leffler property, we have $H^j(\Omega \otimes A)=0$ for $j>0$ where $\Omega :=\varprojlim \Omega _e$.
Note that in fact this works for any Cartier module $\Omega _1:=F^s_*\Omega _0\to \Omega _0$, and hence one can deduce that if $\Omega _0$ is a Cartier module on an abelian variety, then $H^i\big(\Omega \otimes \hat L ^\vee \big)=0$ for $i>0$ \cite[Lem 3.1.2]{HP16}. However, since  $\Omega$ is not a coherent sheaf it is unclear if it satisfies a generic vanishing statement. The situation is clarified in \cite[Thm 3.1.1]{HP16} where we show that if $\Lambda _e=R\hat S(D_A(\Omega _e))$ then $\Lambda :=\hocolim \Lambda _e$ is a sheaf and $\Omega =(-1_A)^*D_A(RS(\Lambda))[-g]$. Despite its technical nature, this statement proves to be very useful in several geometric contexts. Its applications are however somewhat limited in scope since it is not immediately clear  how to recover more traditional generic vanishing statements such as statements on the codimension of the loci $V^i(\omega _X)$ or more functorially on the codimension of the supports of the sheaves $\mathcal H ^i(R\hat S Ra_* \omega _X)$.  
In this paper we take a significant step in this direction. In particular we prove the following.

\subsection{Generic vanishing type results} \ul{From now for the entire article, the base field $k$ is algebraically closed and of characteristic $p>0$.} The Poincar\'e bundle on $A \times \hat A$ is denoted by $\sP$. In particular, $\sP |_{0\times \hat A}=\mathcal O _{\hat A}$ and if $y \in \hat A$ is a closed point then $\sP_y\cong \sP|_{A\times y}$ is the numerically trivial line bundle on $A$ corresponding to $y \in \hat A$.

\begin{theorem}
\label{thm:GV}
 Let $A$ be an abelian variety over $k$ and let $F^s_* \Omega _0 \to \Omega_0$ be a Cartier module on $A$. Then, for each integer $0 \leq i \leq \dim A$, there are closed subsets $ W^i \subseteq \hat A$  such that 
\begin{enumerate}
\item \label{itm:GV:codim} {\scshape codimension:} $\codim W^i \geq i$,
\item \label{itm:GV:zero} {\scshape vanishing:} there is an integer $t>0$ such that for every $i \geq 0$ and for every closed point $y\in \hat A \setminus W^i$, the following homomorphism is zero:
\begin{equation*}
H^i\left(A, \sP_y \otimes F^{t  s}_* \Omega_0\right) \to H^i(A,  \sP_y \otimes \Omega_0),
\end{equation*}
\item \label{itm:GV:limit} {\scshape non-vanishing:} if $Z$ is an irreducible component of $W^i$ of codimension $i$, then for very general closed points $y \in Z$, 
\begin{equation*}
0 \neq \varprojlim_e H^i(A , \sP_y \otimes F^{e  s}_* \Omega_0),
\end{equation*}
\item \label{itm:GV:supersingular_factor} {\scshape torsion property:} if $A$ does not have supersingular factors, then every irreducible component $Z$ of $W^i$ of codimension $i$ is a 
torsion translate of an abelian subvariety,

\item \label{itm:GV:limit_full} {\scshape full limit versions} (direct consequences of points \eqref{itm:GV:zero} and \eqref{itm:GV:limit}): 
\begin{equation*}
y \not\in \bigcup_{j \in \bN} \left[p^{js}\right]^{-1} W^i \textrm{ is a closed point} \quad \Longrightarrow \quad 0 = \varprojlim_e H^i(A , \sP_y \otimes F^{e  s}_* \Omega_0),
\end{equation*}
and if $Z$ is an irreducible component of $\left[p^{js}\right]^{-1} W^i $ of codimension $i$, then
\begin{equation*}
y \in Z \textrm{ a very general closed point} \quad \Longrightarrow \quad 0 \neq \varprojlim_e H^i(A , \sP_y \otimes F^{e  s}_* \Omega_0),
\end{equation*}
\item \label{itm:GV:depth} {\scshape connection to Fourier-Mukai transforms}: independently of $A$ having a supersingular factor  or not, we have:
\begin{equation*}
y \in W^i \setminus \left( \bigcup_{r >i} W^r \right) \quad \Longrightarrow \quad i = \codim y + \depth_{\sO_{\hat A, y}} \tilde \Lambda_{0,y},
\end{equation*}
where for every $ e \gg 0$:
\begin{equation*}
\tilde \Lambda_0:= \im \left( \sH^0\big(R\hat S \big(D_A(\Omega_0)\big) \big) \to \sH^0\big(R\hat S \big(D_A(F^{e  s}_*\Omega_0)\big)\big) \right)
\end{equation*}
and additionally 
\begin{equation*}
\Supp\tilde \Lambda_0 \subseteq \left\{\ \left. y \in \hat A \ \right| \ H^0(A, \sP_y \otimes \Omega_0 )  \neq 0 \ \right\}.
\end{equation*}
In particular, the abelian subvarieties of point \eqref{itm:GV:supersingular_factor} are exactly the closures of the associated primes of $\tilde \Lambda_0$. 
\end{enumerate}
\end{theorem}

Surprisingly, one cannot remove the non-existence of supersingular factors from point \eqref{itm:GV:supersingular_factor} of Theorem \ref{thm:GV}. This is stated in the next proposition, by chosing $y$ of the proposition to be a non-torsion point.

\begin{proposition} (= Example \ref{ex:supersingular_basic})
\label{prop:supersingular_basic}
Point \eqref{itm:GV:supersingular_factor} of Theorem \ref{thm:GV} does not hold if $A$ is supersingular. That is, for each closed point $y \in Y$ of a supersingular abelian variety, there is a Cartier module $F_* \Omega_0 \to \Omega_0$ such that for very general closed point $z \in \hat A$ we have $H^g ( A, \sP_z \otimes \Omega_e)=0$ for every integer $e>0$, and 
 \begin{equation*}
0 \neq \varprojlim_e H^g ( A, \sP_{y} \otimes \Omega_e).
 \end{equation*}
\end{proposition}

\begin{remark}
An admitted aesthetic defect of Theorem \ref{thm:GV} is that the non-vanishing of $\varprojlim_e H^i(A , \sP_y \otimes F^{e  s}_* \Omega_0)$ does not happen for all points  $y \in W^i$, just for very general points of the components of $W^i$ of codimension $i$. Because of the presence of (infinitely many) homomorphisms in these inverse limits, there is not much hope to remove completely the genericity condition from the statement, although it would be interesting to see whether a statement with very general replaced by general can be found. This would go beyond the scope of the present article, so we leave it as an open question.

It is also unclear at the moment whether our statement could be improved to a non-vanishing (at (very) general points of) all components of $W^i$. Let us explain below the reason why this non-vanishing happens only at generic points of components of codimension $i$ in Theorem \ref{thm:GV}. Intuitively, $\varprojlim_e H^i(A , \sP_y \otimes F^{e  s}_* \Omega_0)$ could be zero for two reasons:
\begin{enumerate}
\item \label{itm:as_in_char_zero} the  cohomology group corresponding to $H^i(A , \sP_y \otimes \Omega_0)$ is vanishing in characteristic zero,
\item \label{itm:nilpotent_by_chance} the corresponding cohomology group is not vanishing in characteristic zero, but the action on cohomology induced by $F^s_* \Omega_0 \to \Omega_0$ is still nilpotent because of some particular arithmetic behavior.  
\end{enumerate}
Our approach is based on an inverse system of complexes, which we denote by $ \big(R \hat S \big( \tilde \Omega_e \big) \big)_{e \geq 0}$, and in fact, $W^i$ of Theorem \ref{thm:GV} is defined as $\bigcup_{j \geq i} \Supp R^j \hat S \big( \tilde \Omega_e \big)$. We refer to  Section \ref{sec:Frob_stab_coho_support} for the details, here we focus on the important point that these complexes are able to distinguish points $ y \in \hat A$ of type \eqref{itm:as_in_char_zero} from the rest, but not points of type \eqref{itm:nilpotent_by_chance}. To repair the above defect, one should find an inverse system $(\sC_e)_{e \geq 0}$ of bounded complexes of coherent sheaves, which is equivalent to  $ \big(R \hat S \big( \tilde \Omega_e \big) \big)_{e \geq 0}$ in the sense of Lemma \ref{lem:direct_systems_isomorphic}, and for which the complexes $\sC_e$ can somehow also tell apart points of type \eqref{itm:nilpotent_by_chance} from the rest of the points of $\hat A$. 
Unfortunately, we were not able to find such an inverse system, and hence we leave it also as open question to  decide whether one could strengthen our non-vanishing result beyond the components of codimension $i$. 

\end{remark}

Another bproduct of our methods is the following non-vanishing statement:

\begin{theorem}
\label{thm:non_vanishing}
Let $F^s_* \Omega_0 \to \Omega_0$ be a Cartier module on an abelian variety $A$ without supersingular factors, and let $\hat W \subseteq \hat A$ be an abelian subvariety such that the generic point of a translate of $\hat W$ is an associated prime of $\im \left( R \hat S \big( D_A(\Omega_0) \big) \to R \hat S \big(  D_A(F^{e  s}_* \Omega_0) \big) \right)$ for every $ e \gg 0$ (such $\hat W$ exists for all associated points according to Corollary \ref{cor:torsion_translate_of_ab_var}).  Let $\pi : A \to W$ be the  morphism of abelian varieties dual to the inclusion $\hat W \hookrightarrow \hat A$, and let $j$ be the relative dimension of $\pi$. Then,  $R^j \pi_*( \Omega_0 \otimes Q) \neq 0$ for some $Q \in \hat A$.
\end{theorem}
\begin{remark} It is interesting to note that in characteristic 0, if $\Omega _0$ is a direct summand of the pushforward of the canonical bundle of a smooth variety (or more generally the push forward of a sheaf of the form $\OO _X(D)$ where $D$ is a Weil divisor numerically equivalent to $K_X+B$ for some klt pair $(X,B)$), and if $Z$ is the support of $\Omega _0$, then $R^j\pi_*( \Omega_0 \otimes Q)$ is a torsion free sheaf supported on $\pi (Z)$. It then follows that  $\dim Z _w\geq j$ for every closed point $w\in \pi (Z)$ and hence $Z=\pi ^{-1}\pi (Z)$ i.e. $Z$ is fibred by tori cf. \cite[Theorem 1]{GL90} and \cite[Theorem 3]{EL97}. In order to make further progress in understanding the geometry of varieties in positive characteristics it would be useful to answer the question below.
\end{remark}

\begin{question} Let $f:X\to Y$ be a morphism of smooth projective varieties over an algebraically closed field of characteristic $p>0$. Is it the case  in some sense that:
\begin{enumerate}
\item \label{itm:GR:torsion_free} $ \varprojlim F^e_* R^if_* \omega _X$ is torsion free on $f(Y)$ and in particular
$ \varprojlim F^e_* R^if_* \omega _X=0$ for $i>\dim X-\dim f(X)$, and
\item \label{itm:GR:decomposition} $\holim Rf_*F^e_* \omega _X=\sum \varprojlim F^e_* R^if_* \omega _X[-i]$? 
\end{enumerate}
The words ``some sense''  are added because the ideal framework for the above questions is not entirely clear at this point. For example in case of   \eqref{itm:GR:decomposition}, one could also ask if the equality holds in some category of derived Cartier modules localized by nilpotence. 

\end{question}

\subsection{Geometry of pluri-theta divisors}
By \cite{EL97}, it is known that if $(A,\Theta )$ is a principally polarized abelian variety (PPAV) over $\mathbb C$ the field of complex numbers such that $\Theta$ is irreducible, then $\Theta $ is normal and has rational singularities. In fact $(A,\Theta )$ is plt or equivalently  $\Theta$ has canonical singularities (since it is Gorenstein). We prove an analog of this result in positive characteristic. We recall: \ul{for the entire article, the base field $k$ is algebraically closed and of characteristic $p>0$.} 

\begin{theorem}[Watson]\label{t-main}
Let $(A,X)$ be a PPAV  over $k$ such that $X$ is irreducible and $A$ has no supersingular factors, then $X$ is strongly F-regular.\end{theorem}
\begin{remark} By \cite{Hacon11} we already know that $(A,X)$ is F-pure, however the strongly F-regular statement is substantially more involved. In particular it is natural to ask if it implies Witt rationality. 
\end{remark}
More generally one expects a positive answer to the following.
\begin{question} Let $X$ be a strongly F-regulary variety, then does $X$ have Witt rational singularities?

\end{question}
\begin{remark} We emphasize that several of our techniques are based on and sometimes closely follow the arguments of \cite{Wat16} and as indicated above Theorem \ref{t-main} should be credited to A. Watson. However \cite{Wat16} is yet unpublished and appears to contain several imprecise statements. We hope that the simplifications  that we introduce will help to clarify and disseminate Watson's result.
\end{remark}

\begin{theorem}\label{t-mt} Let $(A,X)$ be a PPAV over $k$ with no supersingular factors  and with $X$  irreducible. Furthermore, let $D\in |mX|$ be a divisor such that $(m,p)=1$ and $\lfloor D/m\rfloor =0$. Then, $(X,D/m)$ is purely F-regular. In particular $(X,D/m)$ is klt and ${\rm mult }_P(D)\leq m j$ for any codimension $j$ point.
\end{theorem}


{
\subsection{$V$-modules} A biproduct of our arguments is a new method of constructing examples of Cartier modules $F^s_* \Omega_0 \to \Omega_0$ on an abelian variety $A$ with a good control on  the loci
\begin{equation}
\label{eq:W_i_F}
W^i_F=\Big\{ \ Q \in \hat A \ \Big| \ \varprojlim_e H^i\big(A,\Omega_0 \otimes Q^{p^e}\big)\ne 0 \ \Big\}.
\end{equation}
We hope that this will lead in the long run to a much better understanding of generic vanishing in positive characteristic in general. The central notion of this approach is the notion of a $V$-module on $\hat A$ which is the abstraction of the sheaf $\tilde \Lambda_0$ of point \eqref{itm:GV:depth} of Theorem \ref{thm:GV}. According to Definition \ref{def:V_module}, it is a coherent sheaf $\sM$ on $\hat A$ together with an $\sO_{\hat A}$-linear morphism $\sM \to V^{s,*} \sM$, where $s>0$ is an integer and $V$ is the Verschiebung isogeny of $\hat A$. The main statement that allows one to run the above mentioned general process of constructing examples of Cartier modules is the following:

\begin{theorem}
\label{thm:V_module_intro} {\scshape (equivalent of Theorem \ref{thm:V_module}})
Let $A$ be an abelain variety over $k$. If $\sM \to V^{s,*} \sM$ is a $V$-module on $\hat A$, then $\varprojlim_e \sH^i\Big( RS\big(D_{\hat A} (  V^{e s,*} \sM )\big) \Big)=0$ for every $i \neq 0$.
\end{theorem}

Using Theorem \ref{thm:V_module_intro}, one can see that the Cartier module $\sH^0\Big( RS\big(D_{\hat A} ( \sM )\big) \Big)$ on $A$ is a Cartier module for which many properties of the associated loci $W^i_F$ can be read off directly from the $V$-module $\sM$.  We refer the reader to Corollary \ref{cor:V_module_comes_from_Cartier_module} for the precise statement of  some of these properties. We believe for example that using Corollary \ref{cor:V_module_comes_from_Cartier_module}, the possibilities for $W^0_F$ can be completely classified, which question we leave for future articles. 
Here, we only state one sample consequence of Corollary \ref{cor:V_module_comes_from_Cartier_module}:

\begin{corollary}
\label{cor:sample_Cartier_module}
Let $L$ be a line bundle on the dual $\hat A$ of an abelian variety $A$ over $k$, and let $W^0_F$ be the $0$-th Frobenius stable cohomology support locus defined in \eqref{eq:W_i_F}.
\begin{enumerate}
\item \label{itm:sample_Cartier_module:supersingular} If $A$ is supersingular, and $0 \neq s \in H^0\big( \hat A, L^{p-1} \big)$, then there is a Cartier module $\Omega_0$ on $A$ such that $W^0_F= \hat A \setminus V(s)$.
\item \label{itm:sample_Cartier_module:no_supersingular_factor} If $A$ has no supersingular factor, and $0 \neq s\in H^0 \big(\hat A, L^{-1} \otimes V^* L \big)$, then there is a Cartier module $\Omega_0$ on $A$ such that 
\begin{equation*}
W^0_F= \hat A \setminus \big\{ \ y \in \hat A \ \big| \ \textrm{there are infinitely many $n \in \bN$ such that $p^n y \in V(s)$} \ \big\}
\end{equation*}
In particular, if $A$ is a surface and $V(s)$ has no component that is a torsion translate of an abelian subvariety, then $\hat A \setminus W^0_F$ is the union of countably many closed points. 
\end{enumerate}
\end{corollary}

We note that as the degree of $V^*L$ is bigger than that of $L$, an $s$ as in point \eqref{itm:sample_Cartier_module:no_supersingular_factor} of Corollary \ref{cor:sample_Cartier_module} can be found for any ample enough line bundle $L$ on $\hat A$. 

We also note that in the surface case of point \eqref{itm:sample_Cartier_module:no_supersingular_factor} of Corollary \ref{cor:sample_Cartier_module} whether $\hat A \setminus W^0_F$ is infinite seems to be connected to the arithmetic behavior of $V(s)$. That is, the only obvious points that are contained in $\hat A \setminus W^0_F$ are the prime-to-$p$ points of $V(s)$, which can be infinite only if either $V(s)$ contains a positive dimensional abelian subvariety or an irreducible component defined over a finite subfield of $k$, by the Manin-Mumford conjecture \cite[Thm 3.6]{PR04}.

Sample questions that our above discussion raises are the following:

\begin{question}
For an abelian variety $A$, for which  $\overline{W^0_F}=\hat A$:
\begin{enumerate}
\item Is $\hat A \setminus W^0_F$ the countable union of constructible subsets?
\item Is the codimension $1$ part of $\hat A \setminus W^0_F$ a  closed, or at least constructible subset?
\item If $A$ has no supersingular factor,  then is every codimension $1$ component of $\hat A \setminus W^0_F$ a torsion translate of an abelian subvariety?
\item If $A$ has no supesingular factor, is then the codimension $2$ part of $\hat A \setminus W^0_F$ not constructible if and only if $V(s)$ contains a component of arithmetic flavor as in the Manin-Mumford conjecture \cite[Thm 3.6]{PR04}?
\end{enumerate}
\end{question}

\subsection{Structure of the article}

The article is structured based on the theorems mentioned in the introduction. 
\begin{itemize}
\item In Section \ref{sec:prelim_and_notation}, we collect some definitions and background statements.
\item In Section \ref{sec:nilpotent_systems}, we collect a few lemmas about the nilpotence behavior of direct and inverse systems in $D^b_c(X)$.
\item In Section \ref{sec:Frob_stab_coho_support}, we show Theorem \ref{thm:GV}.
\item In Section \ref{sec:V_modules_examples}, we show Theorem \ref{thm:V_module_intro}, Corollary \ref{cor:sample_Cartier_module} and Proposition \ref{prop:supersingular_basic}.
\item In Section \ref{sec:max_dimensional_support_loci}, we show Theorem \ref{thm:non_vanishing}.
\item In Section \ref{sec:proof_of_main_thm}, we show Theorem \ref{t-main} and Theorem \ref{t-mt}.
\end{itemize}

\subsection{Acknowledgement}

The first author was partially supported by NSF research grants no: DMS-1952522, DMS-1801851 and by a grant from the Simons Foundation; Award Number: 256202.

The second author was partially supported by the following grants: grant \#200021/169639 from the Swiss National Science Foundation,  ERC Starting grant \#804334.
}

%
%

\section{Preliminaries and notation}\label{S-P}
\label{sec:prelim_and_notation}

\emph{We work over a fixed algebraically closed field $k$ of characteristic $p>0$.} We fix the following notation:
\begin{itemize}
\item $A$ is   a  $g$-dimensional \emph{abelian variety} over  $k$. 
\item $\hat A$ is the \emph{dual abelian variety} of $A$.
\item $\sP$ is the \emph{Poincar\'e line bundle} of $A$, which is a line bundle on $A \times \hat A$.
\item If we write $Q \in \hat A$ for a sheaf $Q$ on $A$, then we mean that $Q = \sP_y$ for some closed point $y \in \hat A$.
\item The \emph{Fourier-Mukai transforms} $R\hat S:D(A)\to D(\hat A)$ and $RS:D(\hat A)\to D(A)$ are defined by
\[ R\hat S(?)=Rp_{\hat A,*}(Lp_A^*?\otimes \mathcal P),\qquad R S(?)=Rp_{A,*}(Lp_{\hat A}^*?\otimes \mathcal P).\]
\item If $L$ is an ample line bundle on $A$, then let $\hat L=R^0\hat S(L)$. In this case, $\phi _L ^*(\hat L ^\vee )\cong \oplus _{h^0(L)}L$, where $\phi_L:A\to \hat A$ is the isogeny defined by $\phi _L( a)=t_{a}^*L^\vee \otimes L$ for any $ a\in  A$. In particular, if $L$ is the line bundle of a principal polarization, then $\deg \phi_L=1$ and $h^0(L)=1$.
\item A \emph{variety} is a scheme of finite type over $k$.
\item As in \cite{Mukai81}, if $X$ is a variety over $k$, then $D(X)$ is the derived category of sheaves of $\sO_X$-modules, and we denote certain full subcategories by putting adequate subscripts and superscripts on $D(X)$. Subscript $c$ and $qc$ means that the cohomology sheaves are required to be coherent and quasi-coherent, respectively. Superscript $b$, $-$ and $+$ means that the cohomology sheaves are zero outside of a finite range, above a finite cohomological degree and below a finite cohomological degree.

We use homotopy-colimits, which are derived category versions of colimits. The only important fact for the article about them is that $\sH^i\big(\hocolim( \_ ) \big)=\varinjlim \big( \sH^i( \_ ) \big)$.

When working with elements of $D(X)$ and its variants, we also use the abbreviation $H^i(X, \_) = R^i \Gamma(X,\_)$.
\item $F$, $V$ and $[p]$ are the usual isogenies of $A$ (Frobenius, Verschiebung and multiplication by $p$). Note that as $k= \overline{k}$, the isogeny $F$ identifies with the absolute Frobenius morphism of $A$. Note also that $F= \hat V$. See \cite[Sec 2.3]{HP16} for further details. 
\end{itemize}
Recall the following.

\begin{theorem}[Mukai \cite{Mukai81}] 
\label{thm:Mukai}
The following equalities  hold in $D_{\rm qc}(A)$ and $D_{\rm qc}(\hat A)$:
\[ R\hat S \circ RS =(-1_{\hat A})^*[-g],\qquad R S \circ R\hat S =(-1_A)^*[-g],\]
\[ D_A\circ RS\cong ((-1_A)^*\circ RS \circ D_{\hat A})[g],\]
\[ RS\circ T_x^*\cong (\otimes P_{-x})\circ RS,\]
\begin{equation*}
\phi^* \circ RS = RS \circ \hat{\phi}_*,
\qquad
\psi_* \circ RS = RS \circ \hat{\psi}^*
\end{equation*}
where $x\in \hat A$, $P_x=\mathcal P |_{A\times x}$, $T_x$ denotes translation by $x$ on $\hat A$,  $[g]$ denotes shift by $g$ spaces to the left, and $\phi$ and $\psi$ are any isogenies of $A$. 
\end{theorem}
\subsection{Truncation functors}
We recall some results from \cite[\S 3.3]{Milicic}.
Let $\mathcal A$ be an abelian category and $ A ^\bullet$ be a complex of objects of $\mathcal A$, then there are two natural truncations given by 
\[ \tau _{\leq n}(A ^\bullet)^p=
\left\{
\begin{array}{ll}
A^p &  {\rm for}\ p<n \\
\ker d^n  &  {\rm for}\ p=n\\
0  &  {\rm for}\  p>n
\end{array}
\right.
\]
\[\tau _{\geq n}
(A ^\bullet)^p=
\left\{
\begin{array}{ll}
0 & {\rm for}\ p<n \\
\coker d^{n-1} & {\rm for}\ p=n \\
A^p & {\rm for}\ p>n.
\end{array}
\right.
\]
These truncations induce functors on the derived category \[\tau _{\geq n}:D(\mathcal A)\to D^-(\mathcal A),\qquad {\rm and} \qquad \tau _{\leq n}:D(\mathcal A)\to D^+(\mathcal A).\]
For every $X \in D(\mathcal A )$, we have morphisms in $D(\mathcal A )$
\[ i:\tau _{\leq n}( X)\to X,\qquad {\rm and}\ q:X\to \tau _{\geq n}( X)\] 
such that the induced maps $\mathcal H^p(X)\to \mathcal H^p( \tau _{\geq n}( X))$ is an isomorphism if $p\geq n$ and zero otherwise and 
$\mathcal H^p(\tau _{\leq n}( X))\to \mathcal H^p(X)$  is an isomorphism if $p\leq n$ and zero otherwise.
For $ A ^\bullet$ as above, there is a short exact sequence of complexes
\[ 0\to  \tau _{\leq n}(A ^\bullet)\to  A ^\bullet\to Q^\bullet \to 0,\]
where $Q^i=0,\ {\rm coim}(d^n),\ A^p,$ if $p<n,\ p=n,\ p>n$. There is an induced projection $Q^\bullet \to\tau _{\geq n+1}(Q^\bullet )=\tau _{\geq n+1}(A^\bullet )$. We have \cite[Proposition 3.6.1]{Milicic} 
\begin{proposition} There exists a unique morphism $h: \tau _{\geq n+1}(A^\bullet )\to \tau _{\leq n}(A^\bullet )[1]$ inducing an exact triangle in $D(\mathcal A)$
\[
\xymatrix{
\tau _{\leq n}(A^\bullet )\ar[r] &  A^\bullet \ar[r] &  \tau _{\geq n+1}(A^\bullet ) \ar[r]^(0.8){+1}_(0.8)h &
}\]\end{proposition} 
\begin{lemma}\label{l-tt} Let $X$ be an object of $D(\mathcal A)$. There is an exact triangle
\[ 
\xymatrix{
\tau _{\leq m}(X)\ar[r] & \tau _{\leq m+1}(X)\ar[r] & \mathcal H^{m+1}(X)[-m-1] \ar[r]^(0.9){+1} &
}
\]
\end{lemma}
\begin{proof} Applying $\tau _{\leq m+1}$ to the morphism $\tau _{\leq m}(X)\to X$ we obtain a morphism $\tau _{\leq m+1}(\tau _{\leq m}(X))\to \tau _{\leq m+1}(X)$. Since  $\tau _{\leq m+1}(\tau _{\leq m}(X))=\tau _{\leq m}(X)$,
we have a triangle $\tau _{\leq m}(X)\to \tau _{\leq m+1}(X)\to Y$. Applying $\mathcal H^\bullet$, since $\mathcal H^p(\tau _{\leq m}(X))\to \mathcal H^p(\tau _{\leq m+1}(X))$ is an isomorphism for $p\ne m+1$, we obtain $\mathcal H^p(Y)=0$ for $p\neq m+1$ and $\mathcal H^{m+1}(Y)=\mathcal H^{m+1}(X)$ so that $Y$ is quasi isomorphic to $\mathcal H^{m+1}(X)[-m-1]$.
\end{proof}

\section{Nilpotent direct systems in the derived category}
\label{sec:nilpotent_systems}

This section is a collection of lemmas about nilpotence behavior of direct and inverse systems in the bounded derived category of coherent sheaves. These statements will be used extensively in Section \ref{sec:Frob_stab_coho_support} and \ref{sec:V_modules_examples}. We will consider arbitrary direct systems of the above type, but for inverse systems we make the assumption that they are associated to (derived) Cartier modules. 
We recall that a {Cartier module} on a scheme $X$ over $k$ is a pair $(\sM, \phi)$, where $\phi$ is an $\sO_X$-homomorphism $F^s_* \sM \to \sM$ for some integer $s>0$. 

\begin{definition}
\label{def:derived_Cartier_modules}
Let $X$ be a variety over $k$. 
An inverse system $(\sD_e, \  \alpha_e: \sD_{e+1} \to \sD_e)_{e\geq 0}$ in $D^b_c(X)$ is \emph{associated to a derived Cartier module}, if there is an integer $s>0$ such that $\sD_e \cong F^{e s}_* \sD_0$ and via these isomorphisms $\alpha_e$  identifies with $ F^{e  s}_* ( \alpha_0)$. 
\end{definition}

\begin{remark}
\label{rem:derived_Cartier_module}
In the situation of Definition \ref{def:derived_Cartier_modules}, 
as $F$ is affine, for each integer $i>0$, the inverse system $(\sH^i(\sD_e), \sH^i(\alpha_e))_{e \geq 0}$ comes from the iterations of a Cartier module. Hence, for every integer $e \geq 0$ the images $\im( \sH^i(\sD_{e'}) \to \sH^i(\sD_e) )$ stabilize for $e' \gg 0$ \cite[Proposition 8.1.4]{BS13}.
\end{remark}

\begin{definition}
Let $X$ be a variety over $k$. 
A direct system $(\sC_e; \alpha_e: \sC_e \to \sC_{e+1})_{e \geq 0}$  in $D^b_c(X)$ is \emph{nilpotent in cohomological degree $i$}, if for every integer $e>0$ the homomorphism $\sH^i(\sC_e) \to \sH^i (\sC_{e'})$ is zero for $e' \gg e$. Note that as $\sH^i(\sC_e)$ is Noetherian, the images $\sH^i(\sC_e) \to \sH^i (\sC_{e'})$ stabilize for $e' \gg 0$, and hence being nilpotent in cohomological degree $i$ is equivalent to $\varinjlim_e \sH^i(\sC_e)=0$. 

The same direct system is \emph{nilpotent  outside of cohomological degree $i$}, if it is nilpotent in each cohomological degree $j \neq i$.

The same direct system is \emph{nilpotent} if for every integer $e>0$ there is an integer $e' \geq e$ such that $\sC_e \to \sC_{e'}$ is zero (where the bigness of $e'$ can depend on $e$). 

Similarly, an inverse system $(\sD_e: \alpha_e: \sD_{e+1} \to \sD_e)$  associated to a derived Cartier module is nilpotent in cohomological degree $i$ if  for every integer $e>0$, the homomorphism   $\sH^i(\sD_{e'}) \to \sH^i(\sD_e) $ is zero for every $e' \gg e$. According to Remark \ref{rem:derived_Cartier_module}, this is equivalent to $\varprojlim_e \sH^i(\sD_e)=0$. Being nilpotent outside of  cohomological degree $i$ and being nilpotent is then defined analogously as above for direct systems. 
\end{definition}

\begin{lemma}
\label{lem:coho_zero_implies_zero}
Let $X$ be a variety over $k$, and let $\big( \sF_e \big)_{\geq 0}$ be a direct system in $D^b_c(X)$, which is nilpotent in all cohomological degrees, and for which  there are integers $c \leq d $ such that $\sH^i(\sF_e) = 0$ for all $e \in \bN$ whenever $i< c$ or $i>d$. Then, $\big( \sF_e \big)_{\geq 0}$ is nilpotent.
\end{lemma}

\begin{proof}
Let $\phi_{e,s}: \sF_e \to \sF_{s}$ be the structural homomorphism.
The proof is by induction on $j:=d-c$. For $d=c$, we have $\sH^c(\sF_e)[-c] \cong \sF_e$. Hence, in this case, for all $s\gg e$ we have that $\phi_{e,s}$ agrees as an arrow of $D(X)$ with the shift by $-c$ of the zero  homomorphism $\sH^c(\sF_e) \to \sH^c(\sF_{s})$, which is zero  in $D(X)$. 

So, let us assume that  $j>0$ and that we know the statement for $j$ replaced by $j-1$. In the rest of the proof we prove the statement for $j$ under these assumptions. Note first that the composition $\sF_e \to \sF_{s} \to \sH^d \left(\sF_{s}\right)[-d]$ agrees with the composition $\sF_e \to  \sH^d \left(\sF_{e}\right)[-d] \to  \sH^d \left(\sF_{s}\right)[-d]$. As the second map is zero in the latter composition for $s \gg e$, we obtain that both compositions are zero in $D(X)$ for $s \gg e$. Fix such an integer $s>e$.
Consider now the distinguished triangle
\begin{equation*}
\xymatrix{
\tau_{\leq d-1} \sF_{s} \ar[r]   & 
\sF_{s} \ar[r] & 
\sH^d \left(  \sF_{s}  \right) [-d] \ar[r]^(0.9){+1} & .
}
\end{equation*}
Applying $\Hom_X\left( \sF_e, \_ \right)$ we obtain an exact sequence, where the image of $\phi_{e,s}$ on the right is zero by the above discussion.
\begin{equation}
\label{eq:coho_zero_implies_zero:exact}
\xymatrix@R=5pt{
\Hom_X\left( \sF_e, \tau_{\leq d-1} \sF_{s}  \right) \ar[r]   & 
\Hom_X\left( \sF_e,\sF_{s} \right) \ar[r] & 
\Hom_X\left( \sF_e, \sH^d \left(  \sF_{s}  \right)  [-d]\right) \\
&
\phi_{e,s} \ar@{|->}[r] \ar@{}[u]|{\rotatebox{90}{$\in$}} & 
0 \ar@{}[u]|{\rotatebox{90}{$\in$}}   .
}
\end{equation}
By the exact sequence \eqref{eq:coho_zero_implies_zero:exact}, we obtain that  $\phi_{e,s}$ descends to an arrow $\widetilde{\phi}_{e,s} : \sF_e \to \tau_{\leq d-1} \sF_{s}$ in $D(X)$. By our induction hypothesis, $\tau_{\leq d-1} \phi_{s,e'} : \tau_{\leq d-1} \sF_{s} \to \tau_{\leq d-1} \sF_{e'}$ is zero for $e' \gg s$. Hence, 
the following diagram then concludes our proof for $e' \gg s$:
\begin{equation*}
\xymatrix@C=70pt@R=10pt{
\ar@/_2.2pc/[drr]_(0.2){\phi_{e,e'}} \sF_e \ar[r]^{\widetilde{\phi}_{e,s}} \ar[dr]|{\phi_{e,s}} & \tau_{\leq d-1} \sF_{s} \ar[d]  \ar[r]^{\tau_{\leq d-1}\phi_{s,e'}=0} & \tau_{\leq d-1} \sF_{e'} \ar[d] \\
& \sF_s \ar[r]|{\phi_{s,e'}} & \sF_{e'}
}
\end{equation*}
\end{proof}

\begin{lemma}
\label{lem:nilpotent_not_nilpotent}
Let $(\sC_e)_{e \geq 0}$, $(\sD_e)_{e \geq 0}$, $(\sE_e)_{e \geq 0}$ be direct systems in $D^b_c(X)$ that form an exact triangle 
\begin{equation}
\label{eq:nilpotent_not_nilpotent:statement}
\xymatrix{
\sC_{\bullet} \ar[r] & \sD_{\bullet} \ar[r] & \sE_{\bullet} \ar[r]^{+1} &
}
\end{equation}
(That is, we have an exact triangle as above for each $\bullet = e \geq 0$, which commute with the structure maps of the systems.) 

Then,
\begin{enumerate}
\item \label{itm:nilpotent_not_nilpotent:nilpotent} if both $(\sC_e)_{e \geq 0}$ and $(\sE_e)_{e \geq 0}$ are nilpotent in cohomological degree $i$, then so is $(\sD_e)_{e \geq 0}$, and
\item \label{itm:nilpotent_not_nilpotent:not_nilpotent} if $(\sE_e)_{e \geq 0}$ is not nilpotent in cohomological degree $i$, and $(\sC_e)_{e \geq 0}$ is nilpotent in cohomological degree $i+1$, then $(\sD_e)_{e \geq 0}$ is not nilpotent in cohomological degree $i$.
\end{enumerate}

\end{lemma}

\begin{proof}
Passing to cohomology sheaves we obtain a long exact sequence of direct systems
\begin{equation*}
\xymatrix{
 \dots \ar[r] & \big(\sH^{i-1}(\sE_e) \big)_{e \geq 0} \ar[r] & 
\big( \sH^{i}(\sC_e) \big)_{e \geq 0} \ar[r] & 
\big( \sH^{i}(\sD_e)\big)_{e \geq 0} \ar[r] & \qquad \qquad \quad \\
& & \qquad \qquad \quad \ar[r] &
\big( \sH^{i}(\sE_e) \big)_{e \geq 0} \ar[r] & 
 \big( \sH^{i+1}(\sC_e) \big)_{e \geq 0} \ar[r] & \dots
}
\end{equation*}
As taking direct limit is exact \cite[\href{https://stacks.math.columbia.edu/tag/04B0}{Tag 04B0}]{stacks-project}, we obtain a long exact sequence
\begin{equation}
\label{eq:nilpotent_not_nilpotent:direct_limit}
\xymatrix{
\dots \ar[r] & \varinjlim_e \sH^{i-1}(\sE_e) \ar[r] & 
\varinjlim_e \sH^{i}(\sC_e) \ar[r] & 
\varinjlim_e \sH^{i}(\sD_e) \ar[r] & \qquad \qquad \quad \\
& & \qquad \qquad \quad \ar[r] &
\varinjlim_e \sH^{i}(\sE_e) \ar[r] & 
\varinjlim_e  \sH^{i+1}(\sC_e) \ar[r] & \dots
}
\end{equation}
As a direct system being nilpotent is equivalent to the corresponding direct limit being zero, we obtain both statements of the lemma directly from \eqref{eq:nilpotent_not_nilpotent:direct_limit}.

\end{proof}

\begin{proposition}
\label{prop:inverse_systems_map}
For a variety $X$ over $k$, let $\big(   \sF_e \big)_{e \geq 0}$ be a direct system of complexes of coherent sheaves, such that:
\begin{enumerate}
\item \label{itm:inverse_systems_map:positive_zero} $\sH^i(\sF_e)=0$ for $i>0$, 
\item  there is $0\geq c \in \bZ$ such that $\sH^i(\sF_e) = 0$ whenever $i< c$, and
\item  \label{itm:inverse_systems_map:limit_zero} $\big( \sF_e\big)_{e \geq 0}$ is nilpotent in cohomological degrees $i < 0$.
\end{enumerate}
Set $\sG_e:= \im \big( \sH^0(\sF_e) \to  \sH^0 \left( \sF_{s} \right) \big)$ for $s \gg e$. Then, for every integer $e>0$, and for every integer $e' \gg e$ (where the bigness depends on $e$), there is a homomorphism as shown by the dashed arrow of the following commutative diagram: 
\begin{equation}
\label{eq:inverse_systems_map:goal}
\xymatrix@C=80pt@R=10pt{
\sF_e \ar[r]^{\phi_{e,e'}} \ar[d]_{\alpha_e} & \sF_{e'} \ar[d]^{\alpha_{e'}} \\
\sG_e \ar[r]_{\eta_{e,e'}}  \ar@{-->}[ur] & \sG_{ e'}
}
\end{equation}
\end{proposition}

\begin{proof}
\fbox{{\scshape Step 1.}} \ul{\emph{It is enough to find a dashed arrow making the following diagram commute:}}
\begin{equation*}
\xymatrix@C=70pt@R=10pt{
\sF_e \ar[r]|{\phi_{e,e'}} \ar[d]_{\alpha_e} & \sF_{e'} \\
\sG_e   \ar@{-->}[ur]_{\xi_{e,e'}} & 
}
\end{equation*}
Indeed, then we have the commutativity of the following two diagrams:
\begin{equation}
\label{eq:inverse_systems_map:step_1_H_0}
\xymatrix@C=80pt{
\sH^0(\sF_e) \ar[r]|{\sH^0\left(\phi_{e,e'}\right)} \ar@{->>}[d]_{\sH^0\left(\alpha_{e}\right)}  & \sH^0(\sF_{e'})  \\
\sG_e   \ar@{-->}[ur]_{\sH^0\left(\xi_{e,e'} \right)}
}
\raisebox{-20pt}{\textrm{, \quad and \quad }}
\xymatrix@C=80pt{
\sH^0(\sF_e) \ar[r]|{\sH^0\left(\phi_{e,e'}\right)}  \ar@{->>}[d]_{\sH^0(\alpha_e)} & \sH^0( \sF_{e'}) \ar@{->>}[d]^{\sH^0(\alpha_{e'})} \\
\sG_e \ar[r]|{\eta_{e,e'}}   & \sG_{ e'}
}
\end{equation}
In particular, the surjectivity of the left arrows of the diagrams of \eqref{eq:inverse_systems_map:step_1_H_0} shows that the bottom triangle of diagram \eqref{eq:inverse_systems_map:step_1_wrap_up} commutes.  Combining this with the commutativity of the upper triangle of  diagram \eqref{eq:inverse_systems_map:step_1_wrap_up}, we obtain the commutativity of \eqref{eq:inverse_systems_map:goal}:
\begin{equation}
\label{eq:inverse_systems_map:step_1_wrap_up}
\xymatrix@C=120pt{
& \sF_{e'} \ar[d]  \ar@/^5pc/[dd]^{\alpha_{e'}} \\
  &  \sH^0(\sF_{e'}) \ar[d]|{\sH^0\left(\alpha_{e'}\right)}  \\
\sG_e  \ar[uur]|{\xi_{e,e'} } \ar[r]  \ar@{-->}[ur]|{\sH^0\left(\xi_{e,e'} \right)} & \sG_{ e'}
}
\end{equation}

\fbox{\scshape Step 2:} \underline{\emph{The cone of $\sF_e \to \sG_e$ is nilpotent in all cohomological degrees:}} for each $e \geq 0$, fix complexes forming the following triangles exact
\begin{equation}
\label{eq:inverse_systems_map:maps_on_K:cone}
\xymatrix@C=50pt{
\sK_e \ar[r]^{\iota_e} & \sF_e \ar[r]^{\alpha_e} & \sG_e \ar[r]^(0.9){+1} &
},
\end{equation}
and fix also arrows $\psi_{e,e+1}$ making the following diagram commute:
\begin{equation*}
\label{eq:inverse_systems_map:maps_on_K}
\xymatrix@C=60pt@R=15pt{
\sK_e \ar[r]^{\iota_e} \ar[d]_{\psi_{e,e+1}} & \sF_e \ar[r]^{\alpha_e} \ar[d]_{\phi_{e,e+1}} & \sG_e \ar[r]^(0.9){+1} \ar[d]_{\eta_{e,e+1}} & \\
\sK_{e+1} \ar[r]_{\iota_{e+1}} & \sF_{e+1} \ar[r]_{\alpha_{e+1}} & \sG_{e+1} \ar[r]^(0.9){+1} &
}
\end{equation*}
This defines an inverse system $( \sK_e)_{e \geq 0}$, and a morphism of inverse systems  $(\iota_e)_{e \geq 0} : ( \sK_e)_{e \geq 0} \to (\sF_e)_{\geq 0}$. Define $\psi_{e,e'}:= \psi_{e'-1, e'} \circ \dots \circ \psi_{e,e+1} : \sK_e \to \sK_{e'}$, and consider the corresponding commutative diagram:
\begin{equation}
\label{eq:inverse_systems_map:maps_on_K}
\xymatrix@C=60pt@R=15pt{
\sK_e \ar[r]^{\iota_e} \ar[d]_{\psi_{e,e'}} & \sF_e \ar[r]^{\alpha_e} \ar[d]_{\phi_{e,e'}} & \sG_e \ar[r]^{+1} \ar[d]_{\eta_{e,e'}} & \\
\sK_{e'} \ar[r]_{\iota_{e'}} & \sF_{e'} \ar[r]_{\alpha_{e'}} & \sG_{e'} \ar[r]^{+1} &
}
\end{equation}
By point \eqref{itm:nilpotent_not_nilpotent:nilpotent} of Lemma \ref{lem:nilpotent_not_nilpotent}, we obtain that $(\sK_e)_{e \geq 0}$ is nilpotent in all cohomology degrees $i$, except posssibly $i=0$ and $i=1$. For $i=1$, we have even $\sH^1(\sK_e)=0$ by assumption \eqref{itm:inverse_systems_map:positive_zero} and by the surjectivity of $\sH^0(\alpha_{e})$. 

So, to finish we are left to show that $(\sK_e)_{e \geq 0}$ is also nilpotent in cohomological degree $i=0$. This is a diagram chase which we explain in the rest of Step 2. For that, denote by upper index $i$ the homomorphisms obtained by  applying $\sH^i(\_)$ to the maps of diagram \eqref{eq:inverse_systems_map:maps_on_K}. Additionally, denote by $\gamma_e^i: \sH^i(\sG_e) \to \sH^{i+1}(\sK_e)$ the edge homomorphisms of the rows of diagram \eqref{eq:inverse_systems_map:maps_on_K}. 
Fix then, $x \in \sH^0(\sK_e)$. We show that $\psi^0_{e,e'}(x)=0$ for $e' \gg e$. Indeed, $\alpha_e^0(\iota_e^0(x))=0$ by the exactness of the rows of \eqref{eq:inverse_systems_map:maps_on_K}. Hence, by the definition of $\sG_e$ this means that $\phi^0_{e,e'}(\iota_e^0(x))=0$ for $e' \gg 0$. However, then by the commutativity of \eqref{eq:inverse_systems_map:maps_on_K}, we obtain that $\iota_{e'}^0(\psi_{e,e'}^0(x))=0$. By the exactness of the rows of \eqref{eq:inverse_systems_map:maps_on_K} and by the fact that $\sH^{-1}(\sG_{e'})=0$, this means that $\psi_{e,e'}^0(x)=0$ holds, which concludes the proof of Step 2.

\fbox{\scshape Step 3:} \ul{\emph{Concluding the argument:}}
According to Step 1, it is enough to lift $ \phi_{e,e'}$ to an element of $\Hom_X(\sG_e, \sF_{e'})$ for $e' \gg e$. This is shown to be true by the following exact sequence:

\begin{equation*}
\xymatrix@R=0pt@C=65pt{
&
\phi_{e, e'} \ar@{}[d]|{\rotatebox{-90}{$\in$}} \ar@{|->}[r] &
\phi_{e,e'} \circ \iota_e 
\explabove{=}{\eqref{eq:inverse_systems_map:maps_on_K} }
\iota_{e'} \circ \psi_{e, e'} 
\explaboveparshift{85pt}{30pt}{=}{$\psi_{e,e'}=0$ by Step 2 and Lemma \ref{lem:coho_zero_implies_zero}}
0
\ar@{}[d]|{\rotatebox{-90}{$\in$}}  \\
%
%
\Hom_X(\sG_e, \sF_{e'})  \ar[r]_{\Hom_X(\alpha_e, \sF_{e'})} & 
\Hom_X(\sF_e, \sF_{e'}) \ar[r]_{\Hom_X(\iota_e, \sF_{e'})} &
 \Hom_X(\sK_e, \sF_{e'}) 
}
\qquad\qquad 
\end{equation*}
\end{proof}

\begin{lemma}
\label{lem:nilpotent_derived_Cartier_module}
Let $X$ be a variety over $k$. If $(\sD_e: \alpha_e: \sD_{e+1} \to \sD_e)_{e \geq 0}$  is an inverse system associated to a derived Cartier module, which is nilpotent in all cohomological degrees, then $\big( \sD_e \big)_{e \geq 0}$ is nilpotent.
\end{lemma}

\begin{proof}
By Definition \ref{def:derived_Cartier_modules}, $\sD_e$ are supported in the same cohomological degrees, say in the interval $[c,d]$, for some integers $c<d$. 
It is enough to show that the direct system $\big(D_X(\sD_e) \big)_{e \geq 0}$ is nilpotent. As  $\big( \sH^i(\sD_e)  \big)_{e \geq 0}$ are nilpotent, we know  that so are the direct systems $\big( D_X(\sH^i(\sD_e)) \big)_{e \geq 0}$. Hence, $\big( D_X(\sH^i(\sD_e)) \big)_{e \geq 0}$ are also nilpotent in every cohomological degree. Applying then point \eqref{itm:nilpotent_not_nilpotent:nilpotent} of Lemma \ref{lem:nilpotent_not_nilpotent} to the following exact triangles inductively for $c \leq j \leq d$ implies that the direct system $\big(D_X(\sD_e)\big)_{e \geq 0}$ is nilpotent in all cohomological degrees. 
\begin{equation*}
\xymatrix{
D_X(\sH^j(\sD_e)) [j] \ar[r] &  D_X(\tau_{\leq {j}} \sD_e) \ar[r] &    D_X(\tau_{\leq {j-1}} \sD_e)  \ar[r]^(0.8){+1} & 
}
\end{equation*}
Lemma \ref{lem:coho_zero_implies_zero} then  concludes our proof. 
\end{proof}

\section{Frobenius stable cohomology support loci}
\label{sec:Frob_stab_coho_support}

In this secton we prove Theorem \ref{thm:GV}.

\begin{notation}
\label{notation:basic}
Let $F^{s}_* \Omega_0 \to \Omega_0$ be a Cartier module on $A$. We set for every integer $e \geq 0$:
\begin{itemize}
\item $\Omega_e:= F^{e  s}_* \Omega_0$,
\item  $\Lambda_e:= R \hat S\big(  D_A( \Omega_e)\big) \cong V^{e s,*} \Lambda_0$, and
\item $\Lambda:=\hocolim_e \Lambda_e$.
\end{itemize}
According to \cite[Cor 3.1.4]{HP16}, $\Lambda$ is supported in cohomological degree $0$, that is, it is quasi-isomorphic to a sheaf, or equivalently $\Lambda \cong \varinjlim_e \sH^0(\Lambda_e)$. In particular, we may set the following notation:
\begin{itemize}
\item $\tilde \Lambda _e={\rm Im}\big(\mathcal H ^0(\Lambda _e)\to \Lambda \big)$,
\item $\tilde \Omega _e:=(-1_A)^*D_ARS\big(\tilde \Lambda _e\big)[-g]=RS\Big(D_{\hat A}\big(\tilde \Lambda _e\big)\Big)$,
\item $\ass_e:= \left\{ \left. \ \xi \in \hat A\ \right|\  \xi\textrm{ is an embedded point of } \tilde \Lambda_e\  \right\}$, and 
\item $\ass:= \displaystyle\bigcup_e \ass_e$. 
\end{itemize}
\end{notation}

Lemma \ref{lem:direct_systems_isomorphic} implies that the inverse systems $\left( R^j \hat{S}( \Omega_e) \right)_{e \geq 0}$ and $\left( R^j \hat{S}\big( \tilde \Omega_e 
\big)\right)_{e \geq 0}$ become naturally isomorphic when passing to the inverse limits.

\begin{lemma}
\label{lem:direct_systems_isomorphic}
In the situation of Notation \ref{notation:basic}, 
there is a monotone increasing sequence $(e_i)_{i \geq 0}$ such that there are morphisms of direct systems given by  $\alpha_i : \Lambda_{e_i}  \to \tilde \Lambda_{e_i}$ and $\beta_i:  \tilde \Lambda_{e_i} \to \Lambda_{e_{i+1}}$, such that both $(\alpha_i)_{i \geq 0 } \circ (\beta_i)_{i \geq 0 }$ and $(\beta_i)_{i \geq 0 } \circ (\alpha_i)_{i \geq 0 }$ are shifts by $1$ in the positive direction of the index $i$. We may also choose $i_0=0$. 
\end{lemma}

\begin{proof}
\cite[Cor 3.1.4]{HP16} tells us that we may apply Proposition \ref{prop:inverse_systems_map} to the direct system $(\Lambda_e )_{e \geq 0}$. That is,  we can define $e_i$ inductively  so that the following diagram commutes:
\begin{equation*}
\xymatrix{
\Lambda_{e_0} \ar[d]|{\alpha_0} \ar[r] & 
\Lambda_{e_1} \ar[d]|{\alpha_1} \ar[r] & 
\Lambda_{e_2} \ar[d]|{\alpha_2} \ar[r] & \dots  \\
%
%
\tilde \Lambda_{e_0}  \ar[r] \ar[ur]|{\beta_0} & 
\tilde \Lambda_{e_1}  \ar[r] \ar[ur]|{\beta_1} & 
\tilde \Lambda_{e_2}  \ar[r] \ar[ur]|{\beta_2} & 
}
\end{equation*}
This concludes our proof. 
\end{proof}

\begin{lemma}
\label{lem:associated_primes_relations}
In the situation of Notation \ref{notation:basic},
 let $\phi : \hat A \to \hat A$ be the multiplication by $p^s$ isogeny. Then: 
\begin{enumerate}
\item \label{itm:associated_primes_relations:preimage}
 $\forall \xi \in \ass_e, \forall \eta \in \phi^{-1}(\xi): \eta \in \ass_{e+1}$,
\item \label{itm:associated_primes_relations:image}
 $\forall \xi \in \ass_{e+1}: \phi(\xi) \in \ass_e$, and
\item \label{itm:associated_primes_relations:embedding} $\forall \xi \in \ass_e: \xi \in \ass_{e+1}$.
\end{enumerate}
\end{lemma}

\begin{proof}
Topologically $V^s$ and $[p]^s$ agree, hence we may redefine $\phi$ to be $V^s$. Then, we have $V^{s,*} \tilde{\Lambda}_e \cong \tilde \Lambda_{e+1}$ \cite[Lem 2.3.1 \& Lem 2.4.5]{HP16}. Additionally, as $V^s$ is faithfully flat, it follows that the emedded points of $\tilde \Lambda_{e+1}$ are exactly the points lying over embedded points of $\tilde \Lambda_e$, via $V^s$ \cite[Prop 6.3.1]{Gro65}.
This yields points \eqref{itm:associated_primes_relations:preimage} and \eqref{itm:associated_primes_relations:image}. Point \eqref{itm:associated_primes_relations:embedding} follows from the fact that we have  natural injection $\tilde \Lambda_e \hookrightarrow \tilde \Lambda_{e+1}$. 
\end{proof}

\begin{corollary}
\label{cor:torsion_translate_of_ab_var}
 In the situation of Notation \ref{notation:basic},  if $A$ has no supersingular factors, then for any integer $e>0$ there exist integers $k_0,k_1>0$ such that for every $\xi \in \ass_e$, there is an abelian subvariety $\hat W \subseteq \hat A$ and a $p^{k_0s} \big(p^{k_1s}-1 \big)$-torsion point $Q \in \hat A$, such that $\overline{\xi}=\hat W + Q$.
\end{corollary}

\begin{proof}
 Fix $\xi \in \ass_e$. 
Points \eqref{itm:associated_primes_relations:image} and \eqref{itm:associated_primes_relations:embedding} of Lemma \ref{lem:associated_primes_relations}, yield that $p^s (\ass_e) \subseteq \ass_e$. 
Since there are only finitely many points in $\ass_e$, and $ \ass_e\supseteq p^s  ( \ass_e) \supseteq p^{2s}  (\ass_e) \ldots$, we may fix $k_0>0$ such that \[p^{k_0s} \ass_e=\cap _{l\geq 1}p^{ls} \ass_e\] and in particular $p^s(p^{k_0s} \ass_e)=p^{k_0s} \ass_e$. We may then pick $k_1>0$ such that  $p^{k_1s} (\xi') = \xi'$ for any $\xi' \in p^{k_0s}\ass_e$ and hence  $p^{(k_1+k_0)s} (\xi) = p^{k_0 s}\xi$ for any $\xi \in \ass_e$.  \cite[Thm 2.3.6]{HP16} yields that  if $\xi \in p^{k_0 s}\ass_e$, then $\overline{p^{k_0 s} \xi} = \hat W + Q$, except that $Q$ might be arbitrary torsion point. However, then 
\begin{equation*}
p^{k_1s} \big(\hat W + Q\big) = p^{k_1s} \left(\ \overline{p^{k_0 s} \xi}\ \right)=\overline{p^{k_0 s} \xi} = \hat W + Q \qquad \Longrightarrow \qquad  (p^{k_1s}-1) Q \in \hat W.
\end{equation*}
Hence, there is $Q' \in \hat W$, such that $(p^{k_1s} -1) Q' = - (p^{k_1s}-1)Q$. For this $Q'$ we have $\hat W + Q = \hat W + (Q + Q')$, and also $(p^{k_1s}-1)(Q + Q')=0$. Therefore, replacing $Q$ by $Q+Q'$ we may assume that $(p^{k_1s}-1)Q=0$. 

Hence, we have $p^{k_0s}\overline\xi  = \overline{p^{k_0s} \xi } =\hat W + Q $, where  $\left(p^{k_1s}-1\right)Q=0$. Thus $\overline\xi=\hat W + R $ where  $(p^{k_1s}-1)p^{k_0s}R=0$.
\end{proof}


\begin{notation}
\label{notation:pt_of_A_hat}
In the situation of Notation \ref{notation:basic},  fix an integer $e \geq 0$ and an arbitrary (scheme theoretic) point $\eta \in \hat A$. Set:
\begingroup
\begin{itemize}
\setlength\itemsep{5pt}
\item $-\eta:=\left(-1_{\hat A}\right)(\eta)$,
\item $j:=\codim \eta$, 
\item  
\renewcommand*{\arraystretch}{1.2}
$d :=
\left\{
\begin{matrix}
 j+1 & \textrm{, if }  \tilde \Lambda_{e,-\eta} = 0 \\
\depth_{\sO_{\hat A, -\eta}} \tilde \Lambda_{e,-\eta}
 & \textrm{, otherwise} \\[5pt]
\end{matrix}
\right.
$
\item $\sP_{\eta}:= \sP|_{A \times \eta}$,
\item $\pr_{\eta} : A \times \eta \to A$
\end{itemize}
\endgroup
\end{notation}

\begin{lemma}\label{lem:main} In the situation of Notation \ref{notation:pt_of_A_hat}, the following holds:
\begin{enumerate}
\item 
\label{itm:main:vanishing}
$R^i \hat S \big(\tilde \Omega_e\big) \otimes k(\eta) = 0$ for every integer $i  > j-d$,
\item \label{itm:main:non_zero}
if $\tilde \Lambda_{e, -\eta} \neq 0$, then $R^{j-d} \hat S \big(\tilde \Omega_e\big) \otimes k( \eta) \neq 0$, and
\item \label{itm:main:limit}
 if $d=0$, then 
$\varprojlim_{r} \Big(R^j\hat S \big( \tilde \Omega _r\big) \otimes k( \eta) \Big) $
 is a non-zero $k(\eta)$-vectorspace.
\end{enumerate}
 
\end{lemma}

\begin{proof} 
\fbox{\scshape Step 1.} \underline{\emph{The proof of points \eqref{itm:main:vanishing} and  \eqref{itm:main:non_zero}:}}
First, note that
\begin{equation}
\label{eq:main:initial}
R\hat S \big(\tilde \Omega _e\big) 
\expl{=}{definition of $\tilde \Omega_e$}
 R\hat S \Big(RS\big(D_{\hat A}\big(\tilde \Lambda _e\big)\big)\Big)
\explshift{-30pt}{=}{Theorem \ref{thm:Mukai}}
 \left(-1_{\hat A}\right)^*D_{\hat A}\big(\tilde \Lambda _e\big)[-g]
\explshift{60pt}{=}{definition of $D_{\hat A}$, and the fact that $\omega_{\hat A}^{\bullet} \cong \sO_{\hat A}[g]$} 
 (-1_A)^*R\sH om \big(\tilde \Lambda _e,\OO _{\hat A} \big).
 \end{equation}
This implies the following, where it is useful to recall that the Matlis duality functor is faithful and exact, and hence the Matlis-dual of a non-zero sheaf is non-zero:
\begin{equation}
\label{eq:main:local_duality}
\left(-1_{\hat A}\right)^*_{\eta} R^i \hat S \big( \tilde \Omega_e \big)_{\eta} 
\expl{\cong}{\eqref{eq:main:initial}} 
\Ext^i_{\sO_{\hat A , -\eta}} \left(  \hat \Lambda_{e,-\eta}, \sO_{\hat A , -\eta} \right) 
\explshift{50pt}{\cong}{this form of local duality is in \cite[Thm A.6]{Schwede_Tucker_A_survey_of_test_ideals}} \Big( H^{j -i}_{m_{\hat A , -\eta}}  \big(  \tilde \Lambda_{e,-\eta} \big) \Big)^{\textrm{ Matlis dual}} 
\end{equation}
Note that $(-1_{\hat A})^*_{\eta}$ is also fully faithful, as $-1_{\hat A}$ is an isomorphism. Hence,
Equation \eqref{eq:main:local_duality} yields directly points \eqref{itm:main:vanishing} and \eqref{itm:main:non_zero} by the characterization of depth via the vanishing/non-vanishing of local cohomology \cite[Exc III.3.4.b]{Hart77}. 

\fbox{\scshape Step 2.} \underline{\emph{If $d =0$, then 
$0 \neq \varprojlim_r \Big(R^j\hat S \big( \tilde \Omega _r\big)_{\eta}  \Big) $:}}
By \eqref{eq:main:local_duality}, it is enough to show that 
\begin{equation}
0 \neq \varinjlim_r  H^{0}_{m_{\hat A , -\eta}}  \big( \tilde \Lambda_{r,-\eta}\big)  
\explshift{-120pt}{=}{direct limits commute with local cohomology} 
 H^{0}_{m_{\hat A , -\eta}}  \bigg(  \varinjlim_r  \tilde \Lambda_{r,-\eta}\bigg)
\explshift{120pt}{=}{direct limits commute with localization}
 H^{0}_{m_{\hat A , -\eta}}  \left(  \Lambda_{-\eta} \right)
\end{equation}
However, this is true, as $\tilde \Lambda_{e, -\eta} \to \Lambda_{-\eta}$ is an injection, and $H^{0}_{m_{\hat A , -\eta}}  \big(\tilde \Lambda_{e, -\eta}  \big) \neq 0$ by the $d=0$ assumption.

\fbox{\scshape Step 3.} \underline{\emph{If $d=0$, then $R^j\hat S \big( \tilde \Omega _r\big)_{\eta}$ is an Artinian $\sO_{\hat A, \eta}$-module for every integer $r \geq 0$:}} for this it is enough to show that $R^j\hat S \big( \tilde \Omega _r\big)_{\eta}$ is supported on $ m _{\hat A, \eta}$ or equivalently that if $\xi \in \hat A$ is a generalization of $\eta$, then $R^j\hat S \big( \tilde \Omega _r\big)_{\xi}=0$. However, in this case $\codim \xi < j$, and hence by applying Step 1 to $\xi$ instead of $\eta$ we obtain this vanishing. 

\fbox{\scshape Step 4.} \ul{\emph{Concluding the argument:}} By Step 3, for every integer $r \geq 0$, there is a unique $\sO_{\hat A, \eta}$-submodule $M_r \subseteq R^j\hat S \big( \tilde \Omega _r\big)_{\eta}$ such that for every $s \gg r$:
\begin{equation*}
M_r = \im \left( R^j\hat S \big( \tilde \Omega _s\big)_{\eta} \to R^j\hat S \big( \tilde \Omega _r\big)_{\eta} \right).
\end{equation*}
Hence, we may find a strictly monotone sequence $0=r_0<r_1 < \dots$, such that we have a commutative diagram as follows
\begin{equation}
\label{eq:main:inverse_diagram_before_tensor}
\xymatrix{
M_{r_0} \ar@{^(->}[d]  & \ar@{->>}[l] 
M_{r_1} \ar@{^(->}[d]  & \ar@{->>}[l]
M_{r_2} \ar@{^(->}[d]  & \ar@{->>}[l] \dots  \\
%
%
R^j\hat S \big( \tilde \Omega _{r_0}\big)_{\eta} & \ar[l] \ar@{->>}[ul]
R^j\hat S \big( \tilde \Omega _{r_1}\big)_{\eta}     & \ar[l]   \ar@{->>}[ul] 
R^j\hat S \big( \tilde \Omega _{r_2}\big)_{\eta} & \ar[l]  \ar@{->>}[ul]  \dots
}
\end{equation} 
Applying $\_ \otimes k(\eta)$ yields the following other commutative diagram, using that tensoring is right-exact:
 \begin{equation}
 \label{eq:main:inverse_diagram_after_tensor}
\xymatrix{
M_{r_0} \otimes k(\eta) \ar[d]  & \ar@{->>}[l] 
M_{r_1} \otimes k(\eta)\ar[d]  & \ar@{->>}[l]
M_{r_2} \otimes k(\eta) \ar[d]  & \ar@{->>}[l] \dots  \\
%
%
R^j\hat S \big( \tilde \Omega _{r_0}\big) \otimes k(\eta)& \ar[l] \ar@{->>}[ul]
R^j\hat S \big( \tilde \Omega _{r_1}\big) \otimes k(\eta)     & \ar[l]   \ar@{->>}[ul] 
R^j\hat S \big( \tilde \Omega _{r_2}\big) \otimes k(\eta) & \ar[l]  \ar@{->>}[ul]  \dots
}
\end{equation} 
As equation \eqref{eq:main:inverse_diagram_before_tensor} yields a relation  as in Lemma \ref{lem:direct_systems_isomorphic} between the inverse systems  $\Big(R^j\hat S\big( \tilde \Omega _r\big)_{\eta}  \Big)_{r \geq 0}$ and $\big( M_r \big)_{r \geq 0}$. Hence, we have  $\varprojlim_r \Big(R^j\hat S \big( \tilde \Omega _r\big)_{\eta}  \Big)\cong \varprojlim_r M_r$.
Combining this with Step 2 we obtain $0 \neq \varprojlim_r M_r$. In particular,  $M_r \neq 0$ for all $r \gg 0$.  Hence, using that $M_{r+1} \otimes k(\eta) \to M_r \otimes k(\eta)$ is surjective, $0 \neq\varprojlim_r \big( M_r \otimes k( \eta)\big)$. Then, as diagram \eqref{eq:main:inverse_diagram_after_tensor} yields also a relation as in Lemma \ref{lem:direct_systems_isomorphic}, we obtain point \eqref{itm:main:limit}.
\end{proof}

For the next statement recall that when working with elements of $D(X)$ and its variants, we also use the abbreviation $H^i(X, \_) = R^i \Gamma(X,\_)$.

\begin{corollary}
\label{cor:coho_and_base-change}
In the situation of Notation \ref{notation:pt_of_A_hat}, we have:
\begin{equation}
\label{eq:coho_and_base-change:statement}
H^i \Big(A \times \eta,  \sP_{\eta} \otimes \pr_{\eta}^* \tilde \Omega_e\Big) 
\cong R^i \hat S \big(\tilde \Omega_e\big) \otimes k(\eta)  
\cong
\left\{
\begin{matrix}
0  & \textrm{if }i > j-d \\
 \neq 0 & \textrm{if }i = j-d\textrm{ and }\tilde \Lambda_{e, - \eta} \neq 0
\end{matrix}
\right.
\end{equation}
\end{corollary}

\begin{proof}
Lemma \ref{lem:main} shows the second isomorphism of \eqref{eq:coho_and_base-change:statement}. So, we are left to show that for $i>j-d$ the left hand side of \eqref{eq:coho_and_base-change:statement} is $0$, and for $i=j-d$ it is isomorphic to $R^i \hat S \big( \tilde \Omega_e \big) \otimes k(\eta)$.
Let $\iota : \eta \to \hat A$,  $j : A \times \eta \to A \times \hat A$  and $q  : A \times \hat A \to A$  be the natural morphisms. Then:
\begin{equation}
\label{eq:coho_and_base-change:derived}
L \iota^* R \hat S \left( \tilde \Omega_e \right) 
\expl{\cong}{\cite[\href{https://stacks.math.columbia.edu/tag/08IR}{Tag 08IR}]{stacks-project}}
 R \Gamma\Big(A \times \eta,   L j^* \big( \sP \otimes q^* \tilde \Omega_e \big) \Big)
\explparshift{260pt}{80pt}{ \cong}{ as $\sP$ is a line bundle and as $q$ and $\pr_{ \eta}$ are flat, we have $L j^* \big( \sP \otimes q^* \tilde \Omega_e \big) \cong (j^* \sP) \otimes L j^* q^* \tilde \Omega_e \cong   \sP_{\eta} \otimes \pr_{\eta}^* \tilde \Omega_e$}
 R \Gamma \Big(A \times \eta,  \sP_{\eta} \otimes \pr_{\eta}^* \tilde \Omega_e \Big)
\end{equation}
According to Lemma \ref{lem:main}, for every integer $i>j-d$, we have $R^i \hat S\big(\tilde \Omega_e \big)_{\eta}=0$. Hence, there is an open neighborhood $ \eta \in U \subseteq \hat A$ and a distinguished triangle
\begin{equation}
\label{eq:coho_and_base-change:triangle}
\xymatrix{
\tau_{< j-d}  R \hat S \big( \tilde \Omega_e \big) |_U \ar[r] & R \hat S \big( \tilde \Omega_e \big) |_U \ar[r] & R^{j-d} \hat S \big(\tilde \Omega_e\big) [d-j]|_U \ar[r]^(0.8){+1} &
}
\end{equation}
Note that as $\tau_{< j-d}  R \hat S \big( \tilde \Omega_e \big)$ is supported in cohomological degrees smaller than $j-d$, equations \eqref{eq:coho_and_base-change:derived}  and  \eqref{eq:coho_and_base-change:triangle} imply that for every integer $i \geq j-d$
\begin{equation}
\label{eq:coho_and_base-change:final}
H^i \Big(A \times \eta,  \sP_{\eta} \otimes \pr_{\eta}^* \tilde \Omega_e\Big)  \cong L^i \iota^* \left( R^{j-d} \hat S \big(\tilde \Omega_e\big) [d-j] \right) = L^{i-(j-d)} \iota ^*\Big( R^{j-d} \hat S \big(\tilde \Omega_e\big) \Big).
 \end{equation}
Using that $L^s \iota^*(\_)$ applied to a sheaf is zero for $s>0$ and it is the same as the application of $(\_) \otimes k(\eta)$ for $s=0$, we obtain from \eqref{eq:coho_and_base-change:final} exactly the description of the left hand side of \eqref{eq:coho_and_base-change:statement} we were looking for. 
\end{proof}

\begin{corollary}
\label{cor:V_i_Omega_tilde_vanishing_non_vanishing}
In the situation of Notation \ref{notation:pt_of_A_hat}, define $W^i:=\bigcup_{r \geq i} \Supp R^r \hat S \big( \tilde \Omega_e \big)$. Then, for every integer $0 \leq i \leq g$ there are the following two options:
\begin{enumerate}
\item \label{itm:V_i_Omega_tilde_vanishing_non_vanishing:not_in_it} if $\eta \notin W^i $, then $H^i \Big(A \times \eta,  \sP_{\eta} \otimes \pr_{\eta}^* \tilde \Omega_e\Big) =0$ and $j-d <i$, and
\item \label{itm:V_i_Omega_tilde_vanishing_non_vanishing:in_it} if $\eta \in W^i \setminus W^{i+1}$, then 
$H^i \Big(A \times \eta,  \sP_{\eta} \otimes \pr_{\eta}^* \tilde \Omega_e\Big) \cong
R^{i} \hat S \big(\tilde \Omega_e\big) \otimes k(\eta) \neq 0$, and $j-d = i$. 
\end{enumerate}
\end{corollary}

\begin{proof}

 One way to read Corollary \ref{cor:coho_and_base-change} is that as we descend with the index $i$, starting from $i=g$, if ever  the cohomology group $H^i \Big(A \times \eta,  \sP_{\eta} \otimes \pr_{\eta}^* \tilde \Omega_e\Big)  $ becomes non-zero, then the first time this happens  at $i=j-d$. Additionally, also by Corollary \ref{cor:coho_and_base-change},  for this value of $i$  we have $\eta \in W^i \setminus W^{i+1}$. This is worded precisely in the present corollary, taken into account that we have $W^g \subseteq W^{g-1} \subseteq \ldots$, and with the added expression for $H^{j-d} \Big(A \times \eta,  \sP_{\eta} \otimes \pr_{\eta}^* \tilde \Omega_e\Big)  $ coming from  Corollary \ref{cor:coho_and_base-change}.


\end{proof}

\begin{proof}[Proof of Theorem \ref{thm:GV}]
In all applications of the statements of the present section we set $e=0$, except for the proof of point \eqref{itm:GV:limit_full}, where this will be stated carefully.
Define $W^i$ as in Corollary \ref{cor:V_i_Omega_tilde_vanishing_non_vanishing}. Set $t$ to be the $e_1$ of Lemma \ref{lem:direct_systems_isomorphic}. That is, we have a factorization
\begin{equation*}
\xymatrix{
\Omega_t =F^{t  s}_* \Omega_0 \ar[r] \ar@/^1pc/[rr] & \tilde \Omega_0 \ar[r] & \Omega_0
}
\end{equation*}
Taking cohomology we obtain homomorphisms
\[ H^i\big(A,\mathcal P _y\otimes F^{t s}_*\Omega_0\big)\to H^i\big(A,\mathcal P _y\otimes \tilde \Omega_0\big)\to H^i\big(A,\mathcal P _y\otimes \Omega_0\big).\]
 By point \eqref{itm:V_i_Omega_tilde_vanishing_non_vanishing:not_in_it} of Corollary \ref{cor:V_i_Omega_tilde_vanishing_non_vanishing}, the middle term vanishes whenever $\eta \in \hat A \setminus W^i$. This implies  point \eqref{itm:GV:zero} of the present theorem.

Point \eqref{itm:V_i_Omega_tilde_vanishing_non_vanishing:in_it} of Corollary \ref{cor:V_i_Omega_tilde_vanishing_non_vanishing} implies that for every $\eta \in W^i \setminus W^{i+1}$ we have $j=\codim \eta \geq i$.  Hence, $i \leq \codim \big( W^i \setminus W^{i+1} \big)$.  Then, by downwards induction on $i$  it follows that $i \leq \codim W^i$. This is the statement of point \eqref{itm:GV:codim} of the present theorem.

Let us assume for the present paragraph that $\eta$ is the generic point of an irreducible  component of $W^i$ of codimension $i$. This is equivalent to assuming that $\codim \eta = i$. Then, as we have already showed point \eqref{itm:GV:codim},  we have $\eta \not\in W^r$ for $r>i$. Hence, by point \eqref{itm:V_i_Omega_tilde_vanishing_non_vanishing:in_it} of Corollary \ref{cor:V_i_Omega_tilde_vanishing_non_vanishing}, we obtain  that $\depth_{\sO_{\hat A, -\eta}} \tilde \Lambda_{0, - \eta}=0$. That is, $- \eta$ is an associated prime of $\tilde \Lambda_0$.  Note now that by Corollary \ref{cor:coho_and_base-change}, cohomology and base-change holds at $\eta$ for $\sF_e:=R^i \pr_* \big( \sP \otimes q^* \tilde \Omega_e \big)$, where $\pr : A \times \hat A \to \hat A$ and $q : A \times \hat A \to A$ are the projections. Then, it also holds for very general closed points of $\overline{\eta}$, for every integer $e \geq 0$. Point \eqref{itm:main:limit} of Lemma \ref{lem:main} also states that $\varprojlim_e \big( \sF_e  \otimes k(\eta) \big)\neq 0$, which then holds for $\eta$ replaced by very general closed points of $\overline{\eta}$ too. That is, we obtain that for very general closed point $y \in \overline{\eta}$ we have $0 \neq \varprojlim_e H^i\big(A, \sP_y \otimes \tilde \Omega_e \big) $. Then, by Lemma \ref{lem:direct_systems_isomorphic} we obtain that $
0 \neq \varprojlim_e H^i\big(A, \sP_y \otimes \Omega_e \big)$ holds too. This concludes the proof of point \eqref{itm:GV:limit}.

To obtain point \eqref{itm:GV:supersingular_factor}, take $\eta$ as in the situation of the previous paragraph. As we have seen there, $-\eta$ is an associated point of $\tilde \Lambda_0$. Hence, by Corollary \ref{cor:torsion_translate_of_ab_var}, we obtain that $\overline{\eta}$ is the torsion translate of an abelian subvariety.

To prove point \eqref{itm:GV:limit_full} of the theorem let us introduce the notations $W^i_e$, which are defined as $W^i$ but for arbitrary $e$ instead of just $e=0$. As $R^j \hat S \big(  \tilde \Omega_e \big) \cong R^j \hat S \big( F^{es}_* \tilde \Omega_0 \big) \cong V^{es,*} R^j \hat S\big(  \tilde \Omega_0\big)$, we have $W^i_e  = V^{e s,*} W^i = \left[p^{es}\right]^{-1} W^i$. Hence, points \eqref{itm:GV:zero} and \eqref{itm:GV:limit}, which we already proved, hold for $\Omega_0$ replaced by $\Omega_e$ if we also replace $W^i$ by $ \left[p^{es}\right]^{-1} W^i$. This $\Omega_e$ version of point \eqref{itm:GV:limit}  yields the non-vanishing part of point \eqref{itm:GV:limit_full} directly. Additionally, by taking inverse limits of the $\Omega_e$ versions of point \eqref{itm:GV:zero} we obtain the vanishing part of point \eqref{itm:GV:limit_full}.

Point \eqref{itm:GV:depth} follows directly from Corollary \ref{cor:V_i_Omega_tilde_vanishing_non_vanishing}, using that being an embedded point of $\Omega_0$ is the same as $\depth_{\sO_{\hat A, - \eta}} \tilde \Lambda_{0, - \eta}$ being zero.

\end{proof}

\section{$V$-modules and examples}
\label{sec:V_modules_examples}

{
The main purpose of this section is to show that many properties of cohomology support loci of Cartier modules on $A$ can be understood by understanding $V$-modules on $\hat A$, which are defined in the next definition.

\begin{definition}
\label{def:V_module}
A \emph{$V$-module} is a pair $(\sM, \phi)$, where $\sM$ is a coherent sheaf on $\hat A$, and $\phi: \sM \to V^{s,*} \sM$ is an $\sO_{\hat A}$-homomorphism for some integer $s>0$. A $V$-module is injective if the structure homomorphism $\phi$ is injective. 
\end{definition}

The main technical statement is the following, which can be thought of as the dual of \cite[Thm 3.1.1]{HP16}.

\begin{theorem}
\label{thm:V_module}
If $\sM \to V^{s,*} \sM$ is a $V$-module, then $\Big( RS \big( D_{\hat A} (  V^{es,*} \sM )\big) \cong F^{es}_* RS\big(D_{\hat A} (\sM )\big)\Big)_{e \geq 0}$ is  an inverse system associated to a derived Cartier module, see Definition \ref{def:derived_Cartier_modules}, and it is nilpotent in all cohomological degrees outside  $0$.
\end{theorem}

\begin{proof}
First we note that the isomorphism $RS \big(D_{\hat A} (  V^{es,*} \sM )\big) \cong F^{es}_* RS \big( D_{\hat A} (\sM)\big)$ is obtained by applying Theorem \ref{thm:Mukai} multiple times:
\begin{multline}
\label{eq:V_module:derived_Cartier_module}
RS \big(D_{\hat A} (  V^{es,*} \sM )\big)
\cong
(-1_A)^* D_A \big(R S(V^{es,*} \sM)\big)[-g]
\cong
(-1_A)^*  D_A \big(F^{es}_*R S( \sM)\big)[-g]
\\[2pt] \cong
F^{es}_* (-1_A)^*  D_A \big(R S( \sM)\big)[-g]
\cong
 F^{es}_* RS \big(D_{\hat A} (\sM)\big).
\end{multline}
Set then:
\begin{equation*}
\sM_e:=V^{es,*} \sM \textrm{, and } \sK_e:=R S ( D_{\hat A}( \sM_e))\cong F^{es}_* \sK_0.
\end{equation*}
As $(\sM_e)_{e \geq 0}$ is associated to a $V$-module, and as the  isomorphisms of \eqref{eq:V_module:derived_Cartier_module} are functorial, $(\sK_e)_{\geq 0}$ is associated to the derived Cartier module $\sK_1 \cong F^s_* \sK_0 \to \sK_0$. 
Note that $\sK_e$ are supported in the same cohomological degrees, which lie in the interval $[-g,0]$. So, assume that $(\sK_e)_{e \geq 0}$ is not nilpotent in all cohomological degrees outside of $0$. Let $j$ be the lowest cohomological degree in which $(\sK_e)_{\geq 0}$ is not nilpotent. By our assumption, to which we would like to find a contradiction, we have $j<0$.
Consider then the following exact triangles for every $i$ and $e$:
\begin{equation}
\label{eq:V_module:main_triangle}
\xymatrix{
R \hat S(D_A(\tau_{\geq i+1} \sK_e)) \ar[r] & R \hat S(D_A(\tau_{\geq i} \sK_e)) \ar[r] &   R \hat S(D_A(\sH^i ( \sK_e)))[i] \ar[r]^(0.9){+1} &
} 
\end{equation}
\emph{We claim that the direct system $\Big(R \hat S\big(D_A(\sH^i ( \sK_e))\big)[i]\Big)_{e \geq 0}$ is nilpotent outside of cohomological degrees (strictly) higher than $-j$.} Indeed, as $\big(\sH^i ( \sK_e)\big)_{\geq 0}$ is an inverse system associated to a Cartier module, by  \cite[Thm 3.1.1]{HP16} we obtain that $\Big(R \hat S\big(D_A(\sH^i ( \sK_e))\big)\Big)_{e \geq 0}$ is nilpotent outside of cohomological degree zero. This solves the claim for $i \geq j$. For $i<j$ on the other hand $\big(\sH^i(\sK_e) \big)_{e \geq 0}$ itself is nilpotent by the choice of $j$, and hence so is $\big(R \hat S(D_A(\sH^i ( \sK_e)))\big)_{e \geq 0}$ in all cohomological degrees. This concludes our claim. 

By induction on $i$, using our above claim, point \eqref{itm:nilpotent_not_nilpotent:nilpotent} of Lemma  \ref{lem:nilpotent_not_nilpotent}  and the exact triangles \eqref{eq:V_module:main_triangle}, we obtain that $\Big( R \hat S \big(D_A(\tau_{\geq i} \sK_e) \big) \Big)_{e \geq 0}$ is nilpotent in cohomogical degree $-j+1$ for all $i$. Additionally, by the arguments with which we proved the above claim, using \cite[Thm 3.1.1]{HP16} again, we see that $\Big(R \hat S \big(D_A(\sH^j ( \sK_e))\big)[j]\Big)_{e \geq 0}$ is not nilpotent in cohomological degree $-j$. Applying \eqref{itm:nilpotent_not_nilpotent:not_nilpotent} of Lemma \ref{lem:nilpotent_not_nilpotent} to the triangle \eqref{eq:V_module:main_triangle} for $i \leq j$ we obtain inductively that $\Big( R \hat S \big(D_A(\tau_{\geq i} \sK_e) \big) \Big)_{e \geq 0}$ is not nilpotent in cohomological degree $-j$ for every integer $i \leq j$. For $i=-g$ this yields that $\Big( R \hat S \big(D_A( \sK_e)\big) \Big)_{e \geq 0}  \cong \big( \sM_e\big)_{e \geq 0}$ is not nilpotent in cohomological degree $-j>0$. This is a contradiction as $\sM$ is a sheaf. 
\end{proof}

The next corollary states that many features of cohomology support loci of Cartier modules on $A$ can be understood by just understanding $V$-modules on $\hat A$, from which the corresponding properties can be read off directly. An example precise statement in Corollary \ref{cor:V_module_comes_from_Cartier_module} is that this works for understanding the behavior of the loci
\begin{equation*}
W^0_F=  \bigg\{ \ Q \in \hat A\  \bigg| \ \varprojlim_e H^0\big(A,\Omega_0 \otimes Q^{p^e} \big) \ne 0\ \bigg\}
\end{equation*}
for any Cartier module $\Omega_0$ on $A$. In fact, Corollary \ref{cor:V_module_comes_from_Cartier_module} collects only some of the properties that can be read off directly from $V$-modules. We hope this list will be further extended in later articles. 

\begin{corollary}
\label{cor:V_module_comes_from_Cartier_module}
If $\sM \to V^{s,*} \sM$ is an injective $V$-module, then we have an embedding of $V$-modules, 
\begin{equation*}
\xymatrix{
\tilde{\Lambda}_0 \ar@{^(->}[r] \ar@{^(->}[d] & V^{s,*} \tilde{\Lambda}_0 = \tilde{\Lambda}_1 \ar@{^(->}[d] \\
\sM \ar@{^(->}[r] & V^{s,*} \sM,
}
\end{equation*}
where $\tilde{\Lambda}_e$ is as in Notation \ref{notation:basic} by setting $\Omega_0:=\sH^0(RS (D_{\hat A}( \sM)))$, and additionally we have $\varinjlim_e V^{es,*} \sM = \varinjlim_e \tilde \Lambda_e$. 

In particular, 
\begin{enumerate}
\item \label{itm:V_module_comes_from_Cartier_module:inverse}
for every closed point $y \in \hat A$ the direct systems $\Big( H^0\big(A,\Omega_0 \otimes \sP_{-p^ey} \big)^{\vee} \Big)_{e \geq 0}$ and $\big(k(y) \otimes V^{es,*} \sM \big)_{e \geq 0}$ are equivalent in the sense of Lemma \ref{lem:direct_systems_isomorphic}. Therefore,  we have
\begin{equation*}
\varprojlim_e H^0\big(A,\Omega_0 \otimes \sP_{-p^ey} \big) \cong  \bigg( \varinjlim_e k(y) \otimes V^{es,*} \sM \bigg)^\vee
\end{equation*}
\item \label{itm:V_module_comes_from_Cartier_module:abelian_subvariety}
if $A$ has no supersingular factors, the abelian subvarieties appearing in Theorem \ref{thm:GV} for the above choice of $\Omega_0$ can be read of  from $\sM$: the embedded points of $\sM$ are torsion translates of exactly these abelian subvarieties. 

\end{enumerate}
\end{corollary}

\begin{proof}
The main task is to prove the part of the statement before ``In particular''. Indeed, assuming everything is proved before ``In particular'', point \eqref{itm:V_module_comes_from_Cartier_module:inverse} follows from \cite[Cor 3.2.1]{HP16}, and point \eqref{itm:V_module_comes_from_Cartier_module:abelian_subvariety} follows from points \eqref{itm:GV:supersingular_factor} and \eqref{itm:GV:depth} of Theorem \ref{thm:GV}. So, we are left to prove the part of the statement before ``In particular''.

Set $\sM_e:= V^{e s,*} \sM$. For each integer $e \geq 0$, we have an exact triangle:
\begin{equation*}
\xymatrix{
\sE_e := \tau_{<0} RS (D_{\hat A}( \sM_e))  \ar[r] & RS (D_{\hat A}( \sM_e)) \ar[r] & \Omega_e = \sH^0( RS (D_{\hat A}(\sM_e))) \ar[r]^(0.8){+1} & 
}
\end{equation*}
where $\sE_e:=F^{es}_* \sE_0$ because of the isomorphism $RS(D_{\hat A}(\sM_e)) \cong F^{es}_* (RS(D_{\hat A}(\sM_0)))$ stated in Theorem \ref{thm:V_module}. In particular, $\sE_e$ is nilpotent in all cohomological degrees by Theorem \ref{thm:V_module}, and hence according to Lemma \ref{lem:nilpotent_derived_Cartier_module}, $\big( \sE_e \big)_{e \geq 0}$ is nilpotent.

Applying $R\hat S (D_A(\_))$ we obtain the exact triangle
\begin{equation}
\label{eq:V_module_comes_from_Cartier_module:exact_triangle_final}
\xymatrix{
\Lambda_e:=R \hat S (D_A( \Omega_e)) \ar[r] & \sM_e \ar[r] & \sC_e:= R \hat S (D_A(  \sE_e)) \ar[r]^(0.8){+1} & 
}
\end{equation}
where  $\big( \sC_e \big)_{e \geq 0}$ is nilpotent. Additionally the triangles \eqref{eq:V_module_comes_from_Cartier_module:exact_triangle_final} form a direct system themselves too as we vary the integer $e \geq 0$. Hence, we have a homomorphism $\big(\sH^0(\Lambda_e) \big)_{e \geq 0} \to \big( \sM_e = \sH^0(\sM_e) \big)_{e \geq 0}$ of direct systems with nilpotent kernel and cokernel. \emph{We claim that this induces an injective homomorphism of direct systems $(\beta_e)_{e \geq 0} :\big(\tilde \Lambda_e\big)_{e \geq 0} \to \big( \sM_e\big)_{e \geq 0}$ with nilpotent cokernel.} Denote by 
 $\psi_{e,\infty} : \sH^0(\Lambda_e) \to \varinjlim_{e'} \sH^0(\Lambda_{e'})$ and $\alpha_e : \sH^0(\Lambda_e) \to \sM_e$ the induced homomorphisms. As the maps $\alpha_e$ induce a homomorphism of direct systems, and as the maps $\sM_e \to \sM_{e'}$ are injective, we obtain that $\ker \psi_{e,\infty} \subseteq \ker \alpha_e $ for every integer $e \geq 0$. However, as $(\ker \alpha_e)_{e \geq0}$ is nilpotent, we have in fact $\ker \psi_{e,\infty} = \ker \alpha_e $. It follows then that there is an induced injection  
 \begin{equation*}
 \beta_e : \tilde \Lambda_e = \factor{\sH^0(\Lambda_e)}{\ker \psi_{e,\infty}} = \factor{\sH^0(\Lambda_e)}{\ker \alpha_e} \to \sM_e
 \end{equation*}
  that forms the injective homomorphism of direct systems as in the claim. 
Additionally, by construction  it follows that $\im \beta_e = \im \alpha_e$. Hence, the cokernels of $\beta_e$ form a nilpotent direct system just as for $\alpha_e$. 
 
Having showed our claim, the direct limit statement of the corollary then follows from the exactness of taking direct limits \cite[\href{https://stacks.math.columbia.edu/tag/04B0}{Tag 04B0}]{stacks-project}.
\end{proof}
\begin{remark} Note that the induced $V$-module is also injective and hence we may iterate the procedure. This gives a decreasing sequence of injective $V$-modules $\tilde \Lambda _0\supset \tilde \Lambda _0^1\supset \tilde \Lambda _0^2\supset \ldots$ with $\varinjlim_e V^{es,*} \sM = \varinjlim_e \tilde \Lambda_e^i$ for any $i>0$. We do not know if this sequence stabilizes i.e. $\cap _{i\geq 0}\tilde \Lambda_0^i=\tilde \Lambda_0^j$ for any $j\gg 0$.
\end{remark}

\begin{proof}[Proof of Corollary \ref{cor:sample_Cartier_module}]
First note that if $A$ is supersingular then $L^{-1} \otimes V^* L \cong L^{-1} \otimes F^* L \cong L^{p-1}$. Hence, we may construct the Cartier module $\Omega_0$ that we use at once for both points.
Set $\sM=L$, and let $\phi : L \to V^* L$ be given by $s$. Then we set $\Omega_0 = \sH^0\Big(RS\big(D_{\hat A}(\sM)\big)\Big)$  as in \autoref{cor:V_module_comes_from_Cartier_module}. We have
 \begin{equation}
 \label{eq:sample_Cartier_module:proof}
 W^0_F
 \explshift{80pt}{=}{point \eqref{itm:V_module_comes_from_Cartier_module:inverse} of Corollary \ref{cor:V_module_comes_from_Cartier_module}}
  \big\{ \ y \in \hat A \ \big| \  k(y) \otimes \varinjlim_e V^{e,*} L \neq 0\ \}
\explshift{40pt}{=}{$V=[p]$ topologically}
 \hat A \setminus \big\{ \   y \in \hat A \  \big| \  p^e(y) \in  V(s) \textrm{ for infinitely many } e \ \big\} 
 \end{equation}
If $A$ is supersingular,  point \eqref{itm:sample_Cartier_module:supersingular} follows diretly from \eqref{eq:sample_Cartier_module:proof} using that in this case topologically $[p]= \Id$. 

For point \eqref{itm:sample_Cartier_module:no_supersingular_factor}, we only have to show the addendum about the case when $A$ is a surface. So, assume that $V(s)$ does not contain any torsion translate of an abelian subvariety. \emph{First we claim that $V(s)$ and $p^e(V(s))$ cannot have common components for all integers $e \gg 0$.} Indeed, otherwise $p$ permutes such components, and hence for every integer $j$ divisible enogh $p^j$ leaves such components fixed. However, then by \cite[Thm 2.3.6]{HP16} such components would be torsion translates of abelian subvarieties. This concludes our claim. Let $e_0$ be the threshold such that for $e \geq e_0$ the claim holds. Then the following computation concludes that when $A$ is a surface without supersingular fators, then $\hat A \setminus W^0_F$ is countable:, 
\begin{equation*}
\hat A \setminus W^0_F 
\expl{\subseteq}{ \eqref{eq:sample_Cartier_module:proof}} 
\bigcup_{j, l \in \bN,]  l \geq j  + e_0} \left[p^j\right]^{-1} V(s) \cap \left[p^l\right]^{-1} V(s)
=
\bigcup_{j, l \in \bN,]  l \geq j  + e_0} \left[p^j\right]^{-1} \Big(
\expl{\underbrace{ V(s) \cap \left[p^{l-j}\right]^{-1} V(s)}}{finite by the above claim}
 \Big).
\end{equation*}

\end{proof}

The following example shows, using $V$-modules, that the torsion property of Theorem \ref{thm:GV} is sharp. That is, it does not hold if $A$ is supersingular. This is the statement  of Proposition \ref{prop:supersingular_basic}.

}

\begin{example}
\label{ex:supersingular_basic}
We use the notation of Corollary \ref{cor:V_module_comes_from_Cartier_module}. Consider a supersingular abelian variety $A$. Then $V: A \to A$ is inseparable. Choose an arbitrary closed point $y \in \hat A$, and set $\sM:= k(y)$. Then, 
\begin{equation*}
V^* \sM=\Spec \left( \factor{\sO_{\hat A}}{m_{\hat A, y}^{[p^g]}} \right).
\end{equation*}
Hence, any element of $x \in \sO_{\hat A, y} \setminus m_{\hat A, y}^{[p^g]}$ such that $m_{\hat A, y} \cdot x \subseteq m_{\hat A,y}^{[p^g]}$ gives a $V$-module structure on $\sM$. Fix any non-zero such element. As for $e=0,1$,  $R S\big(D_{\hat A} (V^{e,*} \sM) \big)$ are unipotent vector bundles, so we have 
\begin{equation*}
\Omega_e = RS\big(D_A( V^{e,*} \sM )\big) \cong \sH^0\Big( RS\big(D_A( V^{e,*} \sM )\big) \Big)\cong  \tilde \Omega_e.
\end{equation*}
It is easy to check that 
\begin{equation*}
\{\  z \in \hat A \ | \ H^g ( A, \sP_z \otimes \Omega_0)\ne 0\ \} = \{-y\}.
\end{equation*}
Additionally by point \eqref{itm:GV:limit} of Theorem \ref{thm:GV} we obtain that
 \begin{equation*}
0 \neq \varprojlim_e H^g ( A, \sP_{-y} \otimes \Omega_e).
 \end{equation*}
\end{example}

\section{Maximal dimensional components of Frobenius stable cohomology support loci}
\label{sec:max_dimensional_support_loci}

 In this section we prove Theorem \ref{thm:non_vanishing}.

\begin{notation}
\label{notation:associated_prime}
We use Notation \ref{notation:basic}, for a Cartier module $F^s_* \Omega_0 \to \Omega_0$ on an abelian variety $A$ without supersingular factors.  Let $\eta \in \hat A$ be an associated prime of $\tilde \Lambda_0$ of codimension $j$. According to Corollary \ref{cor:torsion_translate_of_ab_var}, there exist integers $k_0$ and $k_1$ such that $\overline{\eta} = \hat W + y$, where $y \in \hat A$ is a closed $p^{k_0s}(p^{k_1s}-1)$-torsion point, and $\hat W \subseteq \hat A$ is an abelian subvariety. 
By Lemma \ref{lem:associated_primes_relations}, $\eta':=p^{k_0s}\eta$ is also  an associated prime of $\tilde \Lambda_0$ of codimension $j$ and its closure is $\hat W + y'$  where $y'=p^{k_0s}y$ is a $(p^{k_1s}-1)$-torsion point. Replacing $\eta $ by $\eta'$, we may assume that $\overline{\eta} = \hat W + y$, where $y \in \hat A$ is a closed $(p^{k_1s}-1)$-torsion point.
Using the equation
\begin{equation*}
p^{k_1s} \overline{\eta} = p^{k_1s} \big( \hat W + y \big) = \hat W + p^{k_1s} y 
= \hat W + y = \overline{\eta} \Longrightarrow p^{k_1s} \eta = \eta, 
\end{equation*}
in conjunction with Lemma \ref{lem:associated_primes_relations},  we  define:
\begin{equation*}
\Omega_e':= 
\sP_{y} \otimes \Omega_{k_1e} =  
\sP_{y} \otimes F^{e  k_1s}_* \Omega_0 
\cong F^{e k_1s }_* \left( \sP_{p^{e k_1s } y } \otimes \Omega_0  \right) 
\cong F^{e  k_1s }_*\big( \sP_{y} \otimes \Omega_0 \big)  
= F^{e  k_1s}_* \Omega_0'.
\end{equation*}
That is $F^{k_1s}_* \Omega_0' \to \Omega_0'$ is naturally a Cartier module, such that $F^{e  k_1s }_* \Omega_0' \cong \Omega_e'$. Similarly to Notation \ref{notation:basic} let us introduce also 
\begin{itemize}
\item  $\Lambda_e':= R \hat S\big(  D_A( \Omega_e')\big)$, and
\item $\Lambda':=\hocolim_e \Lambda_e'$.
\item $\tilde \Lambda _e'={\rm Im}\big(\mathcal H ^0(\Lambda _e')\to \Lambda\big)$,
\item $\tilde \Omega _e':=(-1_A)^*D_ARS\big(\tilde \Lambda _e'\big)[-g]=RS\Big(D_{\hat A}\big(\tilde \Lambda _e'\big)\Big)$,
\end{itemize}
In particular, then we have:
\begin{equation*}
\tilde \Lambda_e' \cong T_{-y}^* \tilde  \Lambda_{ek_1}. 
\end{equation*}
Hence, the generic point $\eta'$ of $\hat W$ is an associated prime of $\tilde \Lambda_e'$ for all integers $e \geq 0$. 

Let $W$ be the dual abelian variety of $\hat W$, and let $\pi :A\to W$ be the dual morphism of $\hat W \hookrightarrow \hat A$. Note that $\pi$ is of relative dimension $j$. 
\end{notation}

\begin{assumption}
\label{assumption:zero}
In the situation of Notation \ref{notation:associated_prime}, assume that $R^j \pi_*  \Omega_0' =0$.
\end{assumption}

\begin{claim} 
\label{claim:zero}
Under Assumption \ref{assumption:zero}, we have that 
\begin{equation*} \varinjlim_e \mathcal H^l \left(  RS_{W,\hat W}D_W (R\pi _* \Omega _e') \right)=0\qquad {\rm for}\ l\not\in[-j+1,\ldots,0],
\end{equation*}
or equivalently the direct system $\big( RS_{W,\hat W}D_W (R\pi _* \Omega _e') \big)_{e \geq 0}$ is nilpotent outside of cohomological degrees $\{-j+1,\dots, 0\}$.
\end{claim}
\begin{proof} 
First, note that
\begin{equation}
\label{eq:zero:Omega_e}
R^j\pi _* \Omega _e'
\explshift{-40pt}{=}{Notation \ref{notation:associated_prime}}
 R^j\pi _*\left(F^{e k_1s}_*\Omega _0'\right) 
\expl{\cong}{ $F_W\circ \pi =\pi \circ F_X$ }
 F^{e k_1s}_*\left(R^j\pi _* \Omega _0'\right)
\explshift{40pt}{=}{Assumption \ref{assumption:zero}}
0.
\end{equation}
Then, \emph{we claim that:}
\[ 
\varinjlim_e \mathcal H^i \left(   RS_{W,\hat W}D_W\left(\tau ^{\leq m}R\pi _* \Omega _e'\right)\right)=0\qquad {\rm for}\ i\not\in[-m,\ldots,0].\] 
Since $\tau ^{\leq 0}R\pi _* \Omega _e' =R^0\pi _*\Omega _e' $, the base of the induction holds by \cite[Thm 3.1.1]{HP16}. Consider now the triangle
\begin{equation}\label{e-t}
\xymatrix{
\tau ^{\leq m}R\pi _* \Omega _e'  
\ar[r] & 
\tau ^{\leq m+1}R\pi _* \Omega _e'  
\ar[r] & 
 R^{m+1}\pi _*\Omega _e' [-m-1] \ar[r]^(0.9){+1} & .
 }
\end{equation}
By induction we obtain the following two vanishings:
\[
\forall i\not\in[-m,\ldots,0]  \   :  \    \varinjlim_e  \mathcal H^i \left( RS_{W,\hat W}D_W  \left(\tau ^{\leq m}R\pi _*\Omega _e' \right) \right)=0\] 
\begin{multline*}
\forall i \neq - m-1 \ : \  \varinjlim_e \mathcal H^i \left(  RS_{W,\hat W}
D_W  \left(R^{m+1}\pi _*  \Omega _e'[-m-1] \right)\right)
\\ \cong 
 \varinjlim_e \mathcal H^i \left(  RS_{W,\hat W}
D_W  \left(R^{m+1}\pi _*  \Omega _e' \right)\right)[m+1]
=0.
\end{multline*}
Applying $\mathcal H ^\bullet$ and $\varinjlim$ to the triangle in \eqref{e-t} concludes our claim.

To conclude the proof, we just observe that by \eqref{eq:zero:Omega_e} we have that  $R \pi_* \Omega_e' \cong R \tau^{\leq j-1} \pi_* \Omega_e'$.

\end{proof}

\begin{proof}[Proof of Theorem \ref{thm:non_vanishing}]
We may use Notation \ref{notation:associated_prime} throughout the proof. We have to arrive to contradiction under Assumption \ref{assumption:zero}. 

We start by introducing some additional notation. Consider the morphisms given in the following diagram:
\begin{equation*}
\xymatrix@C=110pt@R=16pt{
& \ar[dl]_{\pr_\eta} A \times \eta \ar[r]^{\tau_\eta} \ar[d]^{j_{\hat W}} & \eta  \ar[d]^{\iota_{\hat W}}  \\
%
%
A \ar[d]_{\pi} & \ar[l]^{\pr_{\hat W}} A \times \hat W \ar[d]^{\rho= \pi \times \id_{\hat W}} \ar[r]_{\tau_{\hat W}} & \hat W  \\
%
%
W & \ar[l]^{\pr_{W \times \hat W}} W \times \hat W \ar[ur]_{\tau_{W \times \hat W}}
} 
\end{equation*}
Set $\mathcal P^{A\times \hat W}:=\mathcal P|_{A\times \hat W}$ and $\sP^{A \times \eta}:= \sP|_{A \times \eta}$, and  notice that
 \begin{equation}
 \label{eq:non_vanishing:Poincare}
 \left.
 \begin{array}{l}
\forall \hat w\in \hat W \subset \hat A \ : \  \mathcal P^{A\times \hat W} |_{A\times \hat w} \cong \rho ^* \mathcal P ^{W\times \hat W} |_{A\times \hat w} \textrm{, and}
 \\[6pt]
 \mathcal P^{A\times \hat W}|_{0_A\times \hat W}=\mathcal O _{\hat W}=\rho^*\mathcal P ^{W\times \hat W}|_{0_A\times \hat W}
 \end{array}
 \right\} 
 \Longrightarrow
 \mathcal P^{A\times \hat W}=\rho^*\mathcal P ^{W\times \hat W}.
 \end{equation}
%
By Claim \ref{claim:zero}, we know that the direct system in $D\big(\hat A \big)$ formed out of the following sheaves is nilpotent in cohomological degree $-j$:
\begin{multline}
\label{eq:non_vanishing:long}
(-1_{\hat W } )^*RS_{W,\hat W}D_{ W}(R\pi _* \Omega _e') 
\expl{\cong}{Identity for exchanging $RS$ and $D$ for $W$, see Theorem \ref{thm:Mukai}}
 D_{\hat W} RS_{W,\hat W}(R\pi _* \Omega _e')[-g+j] 
\\
\expl{\cong}{definition of $RS_{W \hat W}$}
D_{\hat W} R \tau_{W \times \hat W,*} \left( \sP^{W \times \hat W}  \otimes \pr_{W \times \hat W}^* R \pi_* \Omega_e' \right)[-g+j]
\\ \explshift{60pt}{\cong}{flat base-change } 
D_{\hat W} R \tau_{W \times \hat W,*} \left( \sP^{W \times \hat W}  \otimes  R \rho_* \pr_{\hat W}^* \Omega_e' \right)[-g+j]
\explshift{50pt}{\cong}{projection formula, equation \eqref{eq:non_vanishing:Poincare}, and the equation $R \tau_{W \times \hat W,*}  \circ R \rho_* = R \tau_{\hat W,*}$}
D_{\hat W} R \tau_{ \hat W,*} \left( \sP^{A \times \hat W}  \otimes  \pr_{\hat W}^* \Omega_e' \right)[-g+j]
\\ \explshift{40pt}{\cong}{Grothendieck duality}
 R \tau_{ \hat W,*} D_{A \times \hat W} \left( \sP^{A \times \hat W}  \otimes  \pr_{\hat W}^* \Omega_e' \right)[-g+j]
\expl{\cong}{$\omega^{\bullet}_{A \times \hat W} [-(g-j)] \cong \omega^{\bullet}_{A \times \hat W / \hat W}$}
 R\tau_{\hat W,*}  D_{A \times \hat W/\hat W} \left( \sP^{A \times \hat W} \otimes \pr_{\hat W}^* \Omega_e' \right).
\end{multline}
Then, as $\eta \to \hat W$ is flat, the direct system in $D(\hat A)$ formed out of the following sheaves is also nilpotent in cohomological degree $-j$:
\begin{multline*}
\left((-1_{\hat W } )^* RS_{W,\hat W}D_{ W}(R\pi _*  \Omega _e') \right) \otimes k(\eta)
\expl{\cong}{\eqref{eq:non_vanishing:long}} 
\left(R\tau_{\hat W,*}  D_{A \times \hat W/\hat W} \left( \sP^{A \times \hat W} \otimes \pr_{\hat W}^* \Omega_e' \right)
 \right)\otimes k(\eta)
\\
\expl{\cong}{flat base-change}
R\tau_{\eta,*}  j_{\hat W}^* D_{A \times \hat W/\hat W} \left( \sP^{A \times \hat W} \otimes \pr_{\hat W}^* \Omega_e' \right)
\expl{\cong}{localizaiton and duality commutes}
R\tau_{\eta,*}   D_{A \times \eta} \left( \sP^{A \times \eta} \otimes \pr_{\eta}^* \Omega_e' \right)
\\
\expl{\cong}{Grothendieck duality}
D_{k(\eta)}R\Gamma \left(  A \times \eta,  \sP^{A \times \eta} \otimes \pr_{ \eta}^* \Omega_e'	 \right).
\end{multline*}
So, we obtain that 
\begin{equation*}
0=  D_{k(\eta)} \left(  \varinjlim_e  \sH^{-j} \bigg( D_{k(\eta)}R\Gamma \left(  A \times \eta,  \sP^{A \times \eta} \otimes \pr_{ \eta}^* \Omega_e'	 \right) \bigg) \right) 
= \varprojlim_e R^j\Gamma \left(  A \times \eta,  \sP^{A \times \eta} \otimes \pr_{ \eta}^* \Omega_e'	 \right)  
\end{equation*}
However,  by  point \eqref{itm:main:limit} of Lemma \ref{lem:main}, Corollary \ref{cor:coho_and_base-change} we have
\begin{equation*}
0 \neq  \varprojlim_e R^j\Gamma \left(  A \times \eta,  \sP^{A \times \eta} \otimes \pr_{ \eta}^*  \tilde \Omega_e'	 \right) .
\end{equation*}
This contradicts the relation given by  Lemma \ref{lem:direct_systems_isomorphic} between the inverse systems $\left( \tilde \Omega_e' \right)_{e \geq 0}$ and $\left(  \Omega_e' \right)_{e \geq 0}$.
\end{proof}


\section{Proof of Theorem \ref{t-main}}
\label{sec:proof_of_main_thm}

\begin{proof}[Proof of Theorem \ref{t-main}]
Let $X\subset  A$ be a theta divisor, and set $\tau:= \tau(X)$ to be the test ideal of $X$. Since $X$ is a theta divisor, it is reduced.  In particular, this also means that we may assume that $\dim A \geq 2$. 

  Consider then the following diagram of Cartier modules on $A$:
\begin{equation}
\label{eq:main:Cartier_modules}
\xymatrix{
F_*\omega _A \ar@{^(->}[r]^{F_*(a)} \ar[d] & 
F_* \omega_A(X) \ar@{->>}[r]^{F_*(b)} \ar[d] & 
F_* \omega_X \ar[d]  & 
\ar[d] F_* (\omega_X \otimes \tau) \ar@{_(->}[l]^{F_*(\iota)} \\
\omega_A \ar@{^(->}[r]^a & 
\omega_A(X) \ar@{->>}[r]^b & 
\omega_X & 
\omega_X \otimes \tau \ar@{_(->}[l]
}
\end{equation}
Diagram \eqref{eq:main:Cartier_modules}, yields a natural Cartier module structure on $\omega_A(X) \otimes \tau':= b^{-1}(\omega_X \otimes \tau)$, together with the exact sequence
\begin{equation}
\label{eq:main:tau_tau_prime}
0\to \omega _A\stackrel a \to \omega _A(X)\otimes \tau '  \stackrel b \to \omega _X\otimes \tau\to 0.
\end{equation}
%
%
Notice that $X$ is strongly F-regular if and only if $\tau =\OO _X$ or equivalently if $\tau ' = \OO _A$.  Let $\Omega _0:=\omega _X\otimes \tau$. Our goal is to show that $\Omega_0 = \omega_X$.  Proceeding by contradiction we assume that $\Omega_0 \ne \omega_X$. Note that as $X$ is reduced, the cosupport of $\tau$ is a proper closed subset of $X$. 
%
%

From now, we use the additional notations of Notation \ref{notation:basic} given by $\Omega_0$ and $A$. We show several statements next about 
\begin{equation}
\label{eq:main:tilde_Lambda_0}
\Supp \tilde \Lambda_0 
\expl{\subseteq}{\cite[Cor 3.2.1]{HP16}, and the fact that $\Supp \Lambda_0 \supseteq \Supp \tilde \Lambda_0$}
\Big\{\  Q \in \hat A \ \Big| \ H^0\big(A, \Omega_0 \otimes Q^{-1} \big) \neq 0 \ \Big\}.
\end{equation}

\underline{$\Supp \tilde \Lambda_0 \neq \emptyset$:}  as $\tilde \Lambda_e = V^{e,*} \tilde \Lambda_0$, this is equivalent to showing that $\Lambda \neq 0$. This follows from  \cite[Theorem 3.1.1]{HP16} since
 $\Omega =\varprojlim \Omega _e \ne 0$.

\underline{ $\Supp \tilde \Lambda_0 \neq A$:} note that
 $H^0(A, \OO _A(X)\otimes \tau '\otimes Q)=0$ for general $Q\in \hat A$, otherwise the support of every translate of $X$ would contain $Z(\tau ')$, which would imply  that  $Z(\tau ')=\emptyset $. As, $H^1(A, Q)=0$ for $Q \in \hat A \setminus \{\sO_A\}$, using \eqref{eq:main:tau_tau_prime}, we obtain that $H^0(A, \omega_X  \otimes \tau \otimes Q)=0$ for $Q \in \hat A$ general. Then, \eqref{eq:main:tilde_Lambda_0} concludes that $\Supp \tilde \Lambda_0 \neq \hat A$.
 
\underline{ $\dim \Supp \tilde \Lambda_0 >0$:} if not, then $ \tilde \Lambda_e = V^{e,*} \tilde \Lambda_0$ would be Artininan, and \cite[Corollary 3.2.3]{HP16} would imply that $\Supp \Omega = A$, which is not true in the present situation. 

This concludes our statements about $\Supp \tilde \Lambda_0$. As a consequence we  may assume that $\Supp \tilde \Lambda_0$ has an irreducible component $U$,  such that $0< \dim U< g$. According to Corollary \ref{cor:torsion_translate_of_ab_var},   $U =P+\hat W$ where $P$ is a $p^{k_0}(p^{k_1}-1)$-torsion element of $\hat A$ and $\hat W\subset \hat A$ is an abelian subvariety of codimension $0<j <g$. 

Consider the projection $\pi:A\to W$ dual to the inclusion $\hat W\to \hat A$. For any $a\in A$ we let $T_a:A\to A$ be the translation by $a$ and $\lambda (a)=\OO _A(X)^\vee \otimes \OO _{A}(T _a^*X)$ which gives a morphism $\lambda :A\to \hat A$. Since $(A,X)$ is a PPAV, then $\lambda$ is an isomorphism.
Let $\tilde W=\lambda^{-1} \big(\hat W+P\big)\subset A$ be the induced abelian subvariety of $A$ and $\tilde \pi =\pi |_{\tilde W}:\tilde W \to W$ be the induced isogeny.

By Theorem \ref{thm:non_vanishing} we know that $R^j \pi_*( Q \otimes \Omega_0) \neq 0$ for some $Q \in \hat A$. Hence, it is enough to show the following:

\begin{claim} \label{c-v}For every $Q \in \hat A$, we have $R^j\pi _* (Q \otimes \Omega _0)=0$.\end{claim}
\begin{proof}Assume the contrary, that is,  for some fixed $Q \in \hat A$ we have
\begin{equation}
\label{eq:v:assumption}
R^j \pi_* ( Q \otimes \omega_X \otimes \tau) = R^j \pi_*(Q \otimes \Omega_0) \neq 0.
\end{equation}
 Consider the short exact sequence 
$$0\to Q\to Q\otimes \OO _A(X)\to \omega _X\otimes Q\to 0.$$
Pushing forward to $W$, we have $R^{j+1}\pi _* Q=0$, since $\dim (A/W)=j$; as well as  $R^i\pi _*Q(X)=0$ for $i>0$, since $Q(X)$ is ample and so $H^i(A_y,Q(X)|_{A_y})=0$ for any $y\in W$. Thus $R^j\pi _*(\omega _X\otimes Q)=0$.
Now consider the short exact sequence 
\begin{equation}
\label{eq:v:c}
0\to \omega _X\otimes \tau \otimes Q\to \omega _X\otimes Q\to \sC\to 0.
\end{equation}
Pushing forward to $W$, we have  a surjection
\begin{equation}
\label{eq:v:surjection}
R^{j-1}\pi _* 
\expl{\sC}{$\Supp \sC = V(\tau)=:Z$ by \eqref{eq:v:c}}
 \twoheadrightarrow   R^j\pi _*(\omega _X\otimes \tau \otimes Q)
 \expl{\neq}{\eqref{eq:v:assumption}} 
 0.
\end{equation}
 It follows that $Z_{w_0}=Z(\tau)|_{A_{w_0}}$ has codimension $\leq 1$  for some fiber $A_{w_0}$ of $\pi$.

Note that for $R \in (P + \hat W) \setminus \{ \sO_A\}$, we have
\begin{equation*}
H^0(A, \sO_A(X) \otimes \tau' \otimes R^{-1})
\expl{ \twoheadrightarrow}{\eqref{eq:main:tau_tau_prime} and $H^1(A, R) =0$ for $R \in \hat A \setminus \{ \sO_A\}$}
H^0(A, \omega_X \otimes \tau \otimes  R^{-1}) 
\expl{\neq}{\eqref{eq:main:tilde_Lambda_0}}
 0
\end{equation*}
Then, by semicontinuity, using that  $\dim \hat W>0$, it follows that:
\begin{equation*}
\forall R \in P + \hat W : \  H^0\big(A, \sO_A(X) \otimes \tau' \otimes R^{-1}\big) \neq 0 
 \quad \Longrightarrow \quad  \forall  w \in \tilde W: \ H^0\big(A,\OO _A\big(T _{-  w}^* X\big)\otimes \tau ' \big)\ne 0
\end{equation*} 
This means that $T _{w}^* Z$ is contained in the support of $X$ for every $w\in \tilde W$. Since $\tilde W\to W$ is a finite isogeny, it follows that $\tilde W+Z_{w_0}$ is a divisor contained in $X$. Since $X$ is irreducible, $X =\tilde W+Z_{w_0}$ which is not ample (for example because it is disjoint from $\tilde W + z$ for any $z \in A\setminus X$). This is impossible.
\end{proof}
As mentioned before Claim \ref{c-v}, by Theorem \ref{thm:non_vanishing} we know that $R^j \pi_*( Q \otimes \Omega_0) \neq 0$ for some $Q \in \hat A$. Hence, Claim \ref{c-v} yields a contradiction. That is, our assumption that $\tau \neq \sO_X$ is false. Equivalently, $X$ is strongly $F$-regular. 
\end{proof}


\begin{proof}[Proof of Theorem \ref{t-mt}] Choose some integer $e >0$ such that $m | p^e -1$. Consider then the Frobenius trace maps
\[ \Phi ^e:  F^e_* \sO_A(X)\cong F^e_* \OO _A((1-p^e)(K_A+D/m)+p^eX)\to \OO _A(X)\] and let $\tau \otimes \OO _A(X)$ be the image of $\Phi ^e$ for $e\gg 0$. We aim to show that $\tau =\OO _A$. Proceeding by contradiction, assume that $\tau \ne \OO _A$.  Since $\lfloor D/m\rfloor =0$, then $\tau =\mathcal I_Z$ where $Z$ is a subscheme of $A$ of codimension $\geq 2$. We set $\Omega _0 :=\tau \otimes \OO _A(X)$ and we use the associated notation defined in Notation \ref{notation:basic}. As above, we have $\Lambda :={\rm hocolim }(\Lambda _e)=\mathcal H ^0(\Lambda )\ne 0$ and hence $\tilde \Lambda _0\ne 0$. Note that as $X$ does not contain every translate of $Z$, we have $H^0(\tau \otimes \OO _A(X)\otimes P)= 0$ for some (and hence general)  $P\in \hat A$.  Thus $0\leq \dim {\rm Supp}(\tilde \Lambda_0)<\dim A$. 

We will now show that $ \dim {\rm Supp}(\tilde \Lambda_0)\geq 1$. Suppose instead that $ \dim {\rm Supp}(\tilde \Lambda_0)=0$, then $\tilde \Lambda_0$ is Artinian and hence each $\tilde \Omega _e$ is a homogeneous vector bundle (i.e. a successive extension of topologically trivial line bundles). But then since the composition \[ \Omega _t=F^t_*\Omega _0\to \tilde \Omega _0 \to \Omega _0\]
is surjective,  then $F^t_*(\OO _A (X)\otimes \tau)\to \tilde \Omega _0$ is non-trivial.
But as $\tilde \Omega _0$ is a successive extension of topologically trivial line bundles, we must have \[0\ne {\rm Hom}(F^t_*(\OO _A (X)\otimes \tau),P)={\rm Hom}(F^t_*(\OO _A (X')\otimes \tau),\OO _A)\] for some element $P\in {\rm Pic}^0(A)$ and some translate $X'$ of $X$ such that $\OO _A (X')\otimes P^{p^t}=\OO _A (X)$. By Grothendieck duality, we have
\[{\rm Hom}(F^t_*(\OO _A (X')\otimes \tau),\OO _A)=F^t_*{\rm Hom}((\OO _A (X')\otimes \tau),\OO _A).\] But, taking reflexive hulls,  any non-trivial homomorphism $\OO _A (X')\otimes \tau\to \OO _A$ gives rise to a non-trivial homomorphism $\OO _A (X')\to \OO _A$, which is impossible as $H^0(\OO _A(-X') )=0$.

If $\eta$ is a maximal dimensional associated prime of $\tilde \Lambda _0$, then $\bar 
\eta =P+\hat W$ where $P$ is torsion and $\hat W\subset \hat A$ is an abelian subvariety of codimension $j\geq 1$. 
Let $\pi :A \to W$ be the dual projection, $\lambda : A\to \hat A$ be the morphism defined by $a\to \OO _A(T_a^*X-X)$ and $\tilde W=\lambda ^{-1}(\hat W +P)\subset A$,   as in the proof of Theorem \ref{t-main}.
\begin{claim}\label{c-vv} $R^j\pi _*( \Omega _0\otimes Q)=0$ for every $Q\in {\rm Pic }^0(A)$.\end{claim}
\begin{proof}  Suppose the opposite, $R^j\pi _*( \Omega _0\otimes R)\neq 0$ for some $R \in \hat A$.  Consider the short exact sequence
\[ 0\to \OO _A(X)\otimes \tau \otimes R \to  \OO _A(X)  \otimes R\to \QQ  \to 0\]
Since $R^j\pi _* (\OO _A(X) \otimes R) =0$, we have $R^{j-1}\pi _* \QQ\ne 0$. Hence, if $\mathcal I_Z=\tau$, then there exists a fiber $Z_{w_0}$ of dimension $\geq j-1$ for some $w_0\in W$. But then $H^0(A,\OO _A(T_{-w}^*X)\otimes \tau )\ne 0$ for any $w \in \tilde W$ which implies that $X$ contains the non-ample divisor $Z_{w_0}+ \tilde W$. This is impossible as $X$ is irreducible.
\end{proof}

By Claim \ref{c-v}, by Theorem \ref{thm:non_vanishing} we know that $R^j \pi_*( Q \otimes \Omega_0) \neq 0$ for some $Q \in \hat A$. Hence, Claim \ref{c-vv} yields a contradiction. That is, our assumption that $\tau \neq \sO_A$ is false. Equivalently, $(X,D/m)$ is purely $F$-regular. \end{proof}

\end{document}